\documentclass[reqno]{amsart}

\usepackage[normalem]{ulem}

\usepackage{amsmath,amssymb,color}
\usepackage{amsfonts, amscd, epsfig, amsmath, amssymb,enumerate,bm}
\usepackage{graphicx}
\usepackage{graphics}
\usepackage{color}
\usepackage{mathrsfs}
\usepackage{todonotes}
\usepackage{soul}
\usepackage[displaymath, mathlines,pagewise]{lineno}

\usepackage[usenames,dvipsnames]{pstricks}
\usepackage{pstricks-add}
\usepackage{epsfig}
\usepackage{pst-grad} 
\usepackage{pst-plot} 
\usepackage[space]{grffile} 
\usepackage{etoolbox} 
\usepackage[margin=3cm]{geometry}
\makeatletter 
\patchcmd\Gread@eps{\@inputcheck#1 }{\@inputcheck"#1"\relax}{}{}
\makeatother

\newtheorem{theorem}{Theorem}[section]
\newtheorem{assumption}{Assumption}[section]
\newtheorem{remark}{Remark}[section]
\newtheorem{definition}{Definition}[section]

\newtheorem{lemma}[theorem]{Lemma}
\newtheorem{proposition}[theorem]{Proposition}

\usepackage{tikz}
\usetikzlibrary{backgrounds}
\usetikzlibrary{patterns,fadings}
\usetikzlibrary{arrows,decorations.pathmorphing}
\usetikzlibrary{decorations}
\usetikzlibrary{calc}
\usetikzlibrary{shapes.misc}

\usepackage{setspace}

\definecolor{light-gray}{gray}{0.95}
\usepackage{float}
\usepackage[colorlinks=true,linkcolor=blue,citecolor=magenta]{hyperref}
\def\centerarc[#1](#2)(#3:#4:#5){\draw[#1] ($(#2)+({#5*cos(#3)},{#5*sin(#3)})$) arc (#3:#4:#5);}

\newcommand{\vertiii}[1]{{\left\vert\kern-0.25ex\left\vert\kern-0.25ex\left\vert #1 
		\right\vert\kern-0.25ex\right\vert\kern-0.25ex\right\vert}}

\numberwithin{equation}{section}
\numberwithin{figure}{section}

\newcommand{\mc}[1]{{\mathcal #1}}

\newcommand{\bb}[1]{{\mathbb #1}}

\newcommand{\<}{\big\langle}
\renewcommand{\>}{\big\rangle}

\renewcommand{\epsilon}{\varepsilon}

\newcommand{\R}{\mathbb R}
\newcommand{\Z}{\mathbb Z}
\newcommand{\N}{\mathbb N}
\renewcommand{\P}{\mathbb P}

\newcommand{\E}{\mathbb E}

\newcommand{\gen}{\mathscr{L}}

\newcommand{\gens}{\mathscr{S}}

\newcommand{\A}{\mathcal{A}}
\newcommand{\B}{\mathcal{B}}
\newcommand{\Lcal}{\mathcal{L}}
\newcommand{\Scal}{\mathcal{S}}
\newcommand{\kbf}{\mathbf{k}}
\newcommand{\Rbf}{\mathbf{R}}
\newcommand{\Qbf}{\mathbf{Q}}
\newcommand{\qbf}{\mathbf{q}}
\newcommand{\Mbf}{\mathbf{M}}
\newcommand{\mubf}{\bm{\mu}}

\newcommand{\rhobf}{\bm{\rho}}

\newcommand{\etabf}{\bm{\eta}}

\newcommand{\cov}{{\rm Cov}}

\allowdisplaybreaks 

\title[Multi-species ZRP with Long Jumps]{Stationary fluctuations for a multi-species zero range process with long jumps}
\author{Linjie Zhao}
\address{(1) School of Mathematics and Statistics, Huazhong University of Science and Technology, Wuhan 430074, China. (2) Hubei Key Laboratory of Engineering Modeling and Scientific Computing, Huazhong University of Science and Technology, Wuhan 430074, China}
\email{linjie\_zhao@hust.edu.cn}

\thanks{Acknowledgments. The author thanks the financial support from the National Natural Science Foundation of China with grant number 11971038 and the Fundamental Research Funds for the Central Universities.
}

\begin{document}
	
\maketitle

\begin{abstract}
	We consider stationary fluctuations for the multi-species zero range process with long jumps in one dimension, where the underlying transition probability kernel is $p(x) = c_+ |x|^{-1-\alpha}$ if $x > 0$ and $= c_-|x|^{-1-\alpha}$ if $x < 0$. Above, $c_{\pm} \geq 0, \alpha > 0$ are parameters. We prove that for $0 < \alpha < 3/2$, the density fluctuation fields converge to the stationary solution of a coupled fractional Ornstein-Uhlenbeck process, and for $\alpha=3/2$,  the limit points are concentrated on stationary energy solutions to a coupled fractional Burgers equation.   \\
	
\end{abstract}

\section{Introduction}

Over the last years, since the seminal work \cite{gonccalves2014nonlinear}, the universality class from Gaussian to KPZ has been well understood for one dimensional interacting particle systems with one conservation law. Roughly speaking, for weakly asymmetric particle systems with weak asymmetry $N^{-\gamma}$, the stationary fluctuations of the process are usually governed by the Ornstein-Uhlenbeck process if $\gamma > 1/2$, and by stochastic Burgers equation if $\gamma = 1/2$.  This has been verified for a large class of interacting systems, see  \cite{blondel2016convergence,diehl2017kardar,gonccalves2017second} for example.  However, for one dimensional interacting particle systems with several conservation laws, the situation is much more complicated. The mode coupling theory developed in \cite{popkov2015universality,spohn2014nonlinear,spohn2015nonlinear} predicts  different universality classes from Gaussian and KPZ.  However, the analysis is not rigorous and density fluctuations for systems with several conservation laws have only been rigorously proved for a few models, see \cite{bernardin2021derivation,butelmann2022scaling,gonccalves2023derivation,goncalves2022on} and references therein. 

In \cite{gonccalves2018density}, Gon{\c c}alves and Jara considered stationary fluctuations for the exclusion process with long jumps, from which they derived fractional Ornstein-Uhlenbeck process and fractional Burgers equation. Inspired by their work, we aim to extend the result to particle systems with several conservation laws and to derive coupled fractional Ornstein-Uhlenbeck process and coupled fractional Burgers equation. For this purpose, we consider the multi-species zero range process first introduced in \cite{grosskinsky2003stationary}. Stationary fluctuations for this model in the weakly asymmetric case were considered by Bernardin, Funaki and Sethuraman in \cite{bernardin2021derivation}, where they derived coupled stochastic Burgers equation. We focus on the asymmetric case with long jumps. 

In the multi-species zero range process, there are $n \geq 1$ types of particles. For $1 \leq i \leq n$, an $i$--type particle jumps from $x$ to $y$ at rate $g_i (\etabf (x)) c_+ |y-x|^{-1-\alpha}$ if $y > x$, and  at rate $g_i (\etabf (x)) c_- |y-x|^{-1-\alpha}$ if $y < x$. Above, $g_i : \N^n \rightarrow \R_+$ is the jump rate for $i$--type particles, $\etabf (x) = (\eta^i (x))_{1 \leq i \leq n}$ with $\eta^i (x)$ being the number of $i$--type particles at site $x$ and $c_{\pm} \geq 0, \alpha > 0$ are parameters.  We shall prove that when the process starts from its stationary measure, if $0 < \alpha < 3/2$, then the density fluctuation fields converge to a coupled fractional Ornstein-Uhlenbeck process; if $\alpha = 3/2$, then the density fluctuation fields are tight, and any of the limit points is an stationary energy solution to a coupled fractional Burgers equation.  

As far as we know, this is the first attempt to derive coupled fractional Burgers equation from particle systems. Different from particle systems with only one conservation law, the difficulty comes from the fact that the reference frames associated with each conserved quantity are different if there are several conservation laws. Thus, as in \cite{bernardin2021derivation}, we have to assume the frame condition if $\alpha \geq 1$ so that the reference frames are the same for each conserved quantity. Compared with the model in \cite{bernardin2021derivation}, the process with long jumps is superdiffusive if $0 < \alpha < 2$ instead of diffusive. As a consequence,  the estimates for the second order Boltzmann-Gibbs principle,  first proposed in \cite{gonccalves2014nonlinear} as the main tools when considering nonlinear stationary fluctuations, are different.

The rest of the article is organized as follows. In Section \ref{sec:result}, we introduce the model and state stationary fluctuations for the density fluctuation fields. Section \ref{sec:bgp} is devoted to the proof of the second order Boltzmann--Gibbs principle. In Sections \ref{sec:tightness} and \ref{sec:clp} we prove tightness of the density fluctuation fields and characterize the limit points of the fields respectively.    

\section{Notation and Results}\label{sec:result}

In Subsections \ref{subsec:model}, \ref{subsec:sm} and \ref{subsec:sg}, we introduce the model and state its properties including stationary measures and spectral gap estimates. In Subsections \ref{subsec:ou} and \ref{subsec:fbe} we introduce the definition of the solutions to the limiting processes. The main result of this article is stated in Subsection \ref{subsec:sf}.  

\subsection{Multi-species zero range process with long jumps.}\label{subsec:model} The state space of the multi-species zero range process is $\Omega := \big(\N^n\big)^{\Z}$, where  the integer $n\geq 1$ is the number of species of the particles, and $\N = \{0,1,2,\ldots\}$.  For $1 \leq i \leq n$ and $x \in \Z$, denote $\eta^i (x)$ the number of $i$--type particles at site $x$. Then, a configuration $\etabf \in \Omega$ could be written as
\[\etabf = (\etabf (x))_{x \in \Z} = \big((\eta^i(x))_{1 \leq i \leq n}\big)_{x \in \Z}.\]
For $1 \leq i \leq n$, let $g_i: \N^n \rightarrow \R_+:=[0,\infty)$ be the jump rate of  $i$--type particles. To avoid degeneracy, we assume that for any $1 \leq i \leq n$, and for $\kbf=(k^i)_{1 \leq i \leq n} \in \N^n$, 
\[g_i (\kbf)=0 \quad\text{if and only if}\quad k^i = 0, \quad\text{and}\quad g_i^*:= \inf_{k^i \geq 1} g_i (\kbf) > 0.\]

Let $p(\cdot)$ be some transition  kernel on $\Z$ such that $p(0)=0$. We say a function $f: \Omega \rightarrow \R$ is local if the value of $f$ depends only on a finite number of sites. The generator $\gen$ of the process  acts on local functions $f: \Omega \rightarrow \R$ as
\[\gen f (\etabf) = \sum_{x,y \in \Z} \sum_{i=1}^n g_i (\etabf (x)) p(y-x) \big[f(\etabf^{i;x,y}) - f(\etabf)\big], \]
 where the configuration $\etabf^{i;x,y}$ is the one obtained from $\etabf$ after an $i$--type particle jumps from site $x$ to $y$ if there is at least one $i$--type particle at site $x$ in the configuration $\etabf$, \emph{i.e.},  $(\etabf^{i;x,y})^j (z) = \etabf^j (z) + \delta_{i,j} \big(\delta_{z,y} - \delta_{z,x}\big)$ for $1 \leq j \leq n$ and for $z \in \Z$.  As usual, we denote $\delta_{i,j}$ the Delta function.  For $\kbf \in \N^n$, denote \[\kbf^{i,\pm} = (k^1,\ldots,k^{i-1},k^i \pm 1,k^{i+1},\ldots,k^n).\]
 In order for the process to be well defined, we need to assume 
 \[\sup_{1 \leq i,j \leq n} \sup_{\kbf \in \N^n} \big| g_i (\kbf^{j,+}) - g_i (\kbf)\big| < \infty.\]
 We refer the readers to \cite{Andjel82} for construction of the process in the case $n=1$, which could be extended to $n \geq 2$ directly.
 
In this article, we are interested in the process with long jumps (not necessarily symmetric). To this end, let $\alpha \in (0,2)$ and $c_{\pm} \geq 0$ such that $c_+ + c_- > 0$.  Throughout the paper, we take the transition kernel to be
 \[p(z) = \begin{cases}
\frac{c_+}{|z|^{1+\alpha}} , \quad &z > 0;\\
0, \quad &z = 0;\\
\frac{c_-}{|z|^{1+\alpha}} , \quad &z < 0.
 \end{cases}\]
Denote $m:=\sum_{z \in \Z} z p(z)$.

\subsection{Stationary measures.}\label{subsec:sm} In general, the stationary measures for the multi-species zero range process is implicit. As in \cite{grosskinsky2003stationary}, we impose the following conditions on the jump rate function so that the model  has product stationary  measures.

\begin{assumption}\label{assum:stationary}
For any $1 \leq i,j \leq n$ and for  $\kbf=(k^i)_{1 \leq i \leq n} \in \N^n$ such that $k^i,k^{j} \geq 1$, assume
\[g_i (\kbf)g_{j}(\kbf^{i,-}) = g_{j} (\kbf) g_i(\kbf^{j,-}).\]
\end{assumption}

\begin{remark}\label{rmk:G}
	The above assumption is equivalent to the existence of a function $G:\N^n \rightarrow \R$ such that for any $1 \leq i \leq n$, for any $\kbf \in \N^n$ such that $k^i \geq 1$,
	\[\log g_i (\kbf) = G(\kbf) - G(\kbf^{i,-}).\]
Indeed, one could choose
\[G(\kbf) = \sum_{m^1=0}^{k^1} \log g_1 (m^1,0,0,\ldots,0) + \sum_{m^2=0}^{k^2} \log g_2 (k^1,m^2,0,\ldots,0) + \ldots + \sum_{m^n=0}^{k^n} \log g_n (k^1,k^2,\ldots,k^{n-1},m^n).\]
\end{remark}

For  $\mubf = (\mu^i)_{1 \leq i \leq n} \in \N^n$, define the partition function 
\[Z (\mubf) = \sum_{\kbf \in \N^n} \exp \Big\{  -G(\kbf) + \sum_{i=1}^n \mu^i k^i   \Big\}.\]
Let $D_{\mubf}$ be the domain of convergence of the partition function.  For $\mubf \in D_{\mubf}$, define $\bar{\nu}_{\mubf}$ as the product measure on $\Omega$ with marginals
\[\bar{\nu}_{\mubf} \big(\etabf(x) = \kbf \big) = \frac{1}{Z(\mubf)} \exp \Big\{  -G(\kbf) + \sum_{i=1}^n \mu^i k^i   \Big\}, \quad x \in \Z, \; \kbf \in \N^n.\]
Then, the process is stationary with respect to the measure $\bar{\nu}_{\mubf}$.  In order that $D_{\mubf}$ contains a neighborhood of $\mathbf{0}$, we need the following assumption.

\begin{assumption}\label{assump:inv}
Assume
\[\liminf_{|\kbf| \rightarrow \infty} \frac{G(\kbf)}{|\kbf|} > -\infty.\]
\end{assumption}

For $1 \leq i \leq n$, the density of the $i$--type particles with respect to $\bar{\nu}_{\mubf}$ is 
\[R^i (\mubf) = E_{\bar{\nu}_{\mubf}} [\eta^i (x)] = \partial_{\mu^i} \log Z(\mubf) \geq  0.\]
The particle density $\Rbf=(R^i)_{1 \leq i \leq n}: \mathring{D}_{\mubf} \rightarrow D_{\rhobf}$ is well defined, where $D_{\rhobf} = \Rbf (\mathring{D}_{\mubf}) \subset \R_+^n$ is the range of the mapping $\Rbf$. The compressibility is given by 
\[D\Rbf (\mubf) := \Big(\partial_{\mu^j}R^i (\mubf)\Big)_{1 \leq i,j \leq n} = \Big(\partial_{\mu^j}\partial_{\mu^i} \log Z(\mubf)\Big)_{1 \leq i,j \leq n} = \Big( {\rm Cov}_{\bar{\nu}_{\mubf}} \big( \eta^i (0),\eta^j (0) \big)\Big)_{1 \leq i,j \leq n}.\]
Thus, $D\Rbf (\mubf)$ is symmetric and positive definite, which implies $\Rbf$ is invertible on  $\mathring{D}_{\mubf}$. Denote by $\Mbf$ the inverse of $\Rbf$. In order to index the invariant measures by particle densities, for $\rhobf \in D_{\rhobf}$, let $\nu_{\rhobf} = \bar{\nu}_{\Mbf(\rhobf)}$.  For any local function $f: \Omega \rightarrow \R$, let
\[\tilde{f} (\rhobf) = E_{\nu_{\rhobf}} [f].\]


\subsection{Spectral gap.}\label{subsec:sg}  Denote $\gen^*$ the adjoint of the generator $\gen$ in $L^2 (\nu_{\rhobf})$. Then, $\gen^*$ is the generator of a multi-species zero range process with transition kernel $p(-\cdot)$. Let $\gens = (\gen+\gen^*)/2$, which corresponds to a  multi-species zero range process with transition kernel
\[s(x) = \frac{p(x) + p(-x)}{2} = \frac{c_++c_-}{2 |x|^{1+\alpha}}.\]
For integers $\ell > 0$, let $\Lambda_\ell = [-\ell,\ell] \cap \Z$ be the box of length $2\ell+1$ centered at the origin.  For $\kbf \in \N^n$, let $\gens_{\kbf,\ell}$ be the restriction of the generator $\gens$ to the box $\Lambda_\ell$ with $\kbf$ particles, i.e., for function $f: (\N^n)^{\Lambda_\ell} \rightarrow \R$ and for $\etabf$ such that $\sum_{x \in \Lambda_\ell} \etabf(x) = \kbf$,
\[\gens_{\kbf,\ell} f (\etabf) = \sum_{x,y \in \Lambda_\ell} \sum_{i=1}^n g_i (\etabf (x)) s(y-x) \big[f(\etabf^{i;x,y}) - f(\etabf)\big].\]

Since the process with generator $\gens_{\kbf,\ell}$ is irreducible, it has a unique invariant measure, which is denoted by $\nu_{\kbf,\ell}$. Moreover,
\[\nu_{\kbf,\ell} (\cdot)= \nu_{\rhobf} \big(\cdot|\sum_{x \in \Lambda_\ell} \etabf (x) = \kbf \big)\]
does not depend on the density $\rhobf$, and  it is also reversible for the generator $\gens_{\kbf,\ell}$.  Denote $W(\kbf,\ell)$ the inverse of the spectral gap of $-\gens_{\kbf,\ell}$, \emph{i.e.},
\[W(\kbf,\ell)^{-1} = \inf_{f} \frac{D_{\kbf,\ell} (f)}{{\rm Var}_{\nu_{\kbf,\ell}} (f)},\]
where $D_{\kbf,\ell}$ is the canonical Dirichlet form,
\[D_{\kbf,\ell} (f) = \<f(-\gens_{\kbf,\ell})f\>_{\nu_{\kbf,\ell}} = \frac{1}{2}\sum_{x,y \in \Lambda_\ell} \sum_{i=1}^n  s(y-x) E_{\nu_{\kbf,\ell}} \Big[ g_i (\etabf (x)) \big\{f(\etabf^{i;x,y}) - f(\etabf)\big\}^2 \Big].\]
Above, for some probability measure $\mu$, we  use the notation  $\<\cdot \>_\mu = E_{\mu} [\,\cdot\,]$. Throughout the article, we need the following lower bound on the spectral gap.

\begin{assumption}[Spectral gap]\label{assump:sg}
There exists some constant $C=C(\rhobf)$ such that
\[E_{\nu_{\rhobf}} \Big[W\Big(\sum_{x \in \Lambda_\ell} \etabf (x),\ell\Big)^2\Big] \leq C \ell^{2\alpha}.\]
\end{assumption} 

\begin{remark}
	A sufficient condition for the above assumption to hold is  
	\[\inf_{1 \leq i \leq n} \inf_{\kbf \in \N^n} \big\{ g_i (\kbf^{i,+}) - g_i (\kbf)\big\} > \varepsilon_0\]
	for some $\varepsilon_0 > 0$. Indeed,  by the proof of \cite[Lemma 2.2]{bernardin2021derivation} and \cite[Theorem 1.5]{jara2018spectral}, one could show that there exists some constant $C$ such that
	\[W(\kbf,\ell) \leq C \ell^\alpha\]
	uniformly in $\kbf \in \N^n$.  
\end{remark}

\subsection{The Ornstein-Uhlenbeck process}\label{subsec:ou} Denote $\mc{S} (\R)$ the usual Schwartz space and $\mc{S}^\prime (\R)$ the space of tempered distributions. For $H \in \mc{S} (\R)$, define the operator $\Lcal: \mc{S} (\R) \rightarrow L^2 (\R)$ as
\begin{equation}\label{l}
	\Lcal H (u) = \int_{\R} \frac{c_+ \mathbf{1}_{\{v > 0\}} + c_-\mathbf{1}_{\{v < 0\}}}{|v|^{1+\alpha}} \Big(H (u+v) - H (u) - \theta^\alpha (v) H' (u)\Big) dv,
\end{equation}
where
\[\theta^\alpha (v) = \begin{cases}
	0, \quad &\alpha < 1;\\
	v\mathbf{1}_{\{|v|\leq 1\}}, \quad &\alpha=1;\\
	v,\quad &\alpha >1.
\end{cases}\]
Denote by $\Lcal^*$ the adjoint of $\Lcal$ in $L^2 (\R)$ and let $\Scal = (\Lcal + \Lcal^*) /2$. More precisely, for  
$H \in \mc{S} (\R)$,
\begin{align*}
	\Lcal^* H (u) &= \int_{\R} \frac{c_+ \mathbf{1}_{\{v < 0\}} + c_-\mathbf{1}_{\{v > 0\}}}{|v|^{1+\alpha}} \Big(H (u+v) - H (u) - \theta^\alpha (v) H' (u)\Big) dv,\\
	\Scal H (u) &= \frac{c_++c_-}{2}\int_{\R} \frac{1}{|v|^{1+\alpha}} \Big(H (u+v) - H (u) \Big) dv.
\end{align*}
Note that the operator $\Scal$ is a constant multiple of the fractional Laplacian $-(-\Delta)^{\alpha/2}$.

\begin{remark}\label{rmk:L}
We underline that for general $H \in \mc{S} (\R)$, $\Lcal H$ may not belong to $\mc{S} (\R)$. However, by \cite[Proposition 2.4]{gonccalves2018density}, the operator $\Lcal: \mc{S} (\R) \rightarrow L^2 (\R)$ is continuous. Moreover, for $H \in \mc{S} (\R)$, $\Lcal H$ is bounded and infinitely differentiable. 
\end{remark}

Throughout  the article, we fix a time horizon $T > 0$. For $H,G \in L^2(\R)$, denote
\[\<H,G\> = \int_\R H(u) G(u) du.\]
Denote by $\dot{\mc{W}}_t = \big(\dot{\mc{W}}_t^i\big)_{1 \leq i \leq n}$ the standard $\mc{S}' (\R)^n$--valued space--time white noise with covariance function given by
\[\cov \Big( \dot{\mc{W}}_t^i (H), \dot{\mc{W}}_s^j (G)\Big) = \delta_{i,j} \delta_{s,t} \<H,G\> \]
for any $1 \leq i,j \leq n$, any $0 \leq s,t \leq T$ and any $H,G \in \mc{S} (\R)$. Define
\begin{align*}
	\Qbf (\rhobf) &= \big(Q_{i,j} (\rhobf)\big)_{1 \leq i,j \leq n} = \big( (\partial_{\rho^{j}} \tilde{g}_i) (\rhobf) \big)_{1 \leq i,j \leq n},\\
	\mathbf{P} (\rhobf)  &= {\rm diag} \big( (\partial_{\rho^{i}}\tilde{g}_i) (\rhobf) \big)_{1 \leq i \leq n},\\
	\qbf (\rhobf) &= {\rm diag} \big( \sqrt{\tilde{g}_i (\rhobf)}\big)_{1 \leq i \leq n}.
\end{align*}

Next, we introduce the definition of the stationary solution to a family of fractional Ornstein-Uhlenbeck processes. We say  the $\mc{S}' (\R)^n$--valued process $\{\mc{Y}_t, 0 \leq t \leq T\}$ is stationary if for any time $0 \leq t \leq T$, $\mathcal{Y}_t$ is Gaussian  with covariance, for any $1 \leq i,j \leq n$ and any $H,G \in \mc{S} (\R)$,
\[\cov\big( \mc{Y}^i_t (H), \mc{Y}_t^j (G) \big) = \Gamma_{i,j} (\rhobf) \<H,G\>\]
where $\Gamma_{i,j} (\rhobf) = E_{\nu_{\rhobf}} [(\eta^i (0) - \rho^i) (\eta^j (0) - \rho^j)]$. For any trajectory $H: [0,T] \rightarrow \mc{S} (\R)$, we also denote $H_s (\cdot) = H(s,\cdot)$. 

\begin{definition}\label{def:ou} 
{\rm (1)  (The case $0 < \alpha < 1$).} We say $\{\mc{Y}_t, 0 \leq t \leq T\}$ is a stationary solution of the equation
\begin{equation}\label{ou1}
	\partial_t \mc{Y}_t = \Qbf (\rhobf) \Lcal^* \mc{Y}_t + 2 \qbf (\rhobf) \sqrt{-\Scal} \dot{\mc{W}}_t, 
\end{equation}
or componentwise, for any $1 \leq i \leq n$,
\[\partial_t \mc{Y}_t^i =  \sum_{j=1}^{n} (\partial_{\rho^{j}} \tilde{g}_i) (\rhobf) \Lcal^* \mc{Y}^j_t + 2  \sqrt{\tilde{g}_i (\rhobf) (-\Scal)} \dot{\mc{W}}^i_t, \]
if the following conditions are satisfied:
\begin{itemize}
	\item the process  $\{\mc{Y}_t, 0 \leq t \leq T\}$ is stationary;
	\item for any $1 \leq i \leq n$ and for any trajectory $H: [0,T] \rightarrow \mc{S} (\R)$,
	\begin{equation}\label{msub1}
		\mc{M}_t^{i} (H) = \mc{Y}_t^{i} (H_t) - \mc{Y}_0^{i} (H_0)  - \int_0^t \sum_{j=1}^n (\partial_{\rho^{j}} \tilde{g}_i) (\rhobf) \mc{Y}_s^{j} ((\partial_s + \Lcal) H_s) \,ds
	\end{equation}
	is a continuous martingale with quadratic variation 
	\[4 \tilde{g}_i (\rhobf) \int_0^t \<H_s,(-\Scal) H_s\> ds.\]
	Moreover, for $1 \leq i \neq j \leq n$, the martingales $\mc{M}_\cdot^{i}$ and $\mc{M}_\cdot^{j}$ are independent.
\end{itemize}

{\rm (2)  (The case $1 \leq  \alpha < 3/2$).} We say $\{\mc{Y}_t, 0 \leq t \leq T\}$ is a stationary solution of the equation
\begin{equation}\label{ou2}
	\partial_t \mc{Y}_t = \mathbf{P}  (\rhobf)  \Lcal^* \mc{Y}_t + 2 \qbf (\rhobf) \sqrt{-\Scal} \dot{\mc{W}}_t, 
\end{equation}
or componentwise, for any $1 \leq i \leq n$,
\[\partial_t \mc{Y}_t^i =   (\partial_{\rho^{i}} \tilde{g}_i) (\rhobf) \Lcal^* \mc{Y}^i_t + 2  \sqrt{\tilde{g}_i (\rhobf) (-\Scal)} \dot{\mc{W}}^i_t, \]
if the martingale in \eqref{msub1} is replaced with 
\begin{equation}
\mc{M}_t^{i} (H) = \mc{Y}_t^{i} (H_t) - \mc{Y}_0^{i} (H_0)  - (\partial_{\rho^{i}} \tilde{g}_i) (\rhobf) \int_0^t  \mc{Y}_s^{i} ((\partial_s + \Lcal) H_s) \,ds.
\end{equation}
Note that in this case, the processes $\{\mc{Y}_t^i\}$, $1 \leq i \leq n$, are decoupled.
\end{definition}

\begin{remark}
By direct calculations, for $H \in \mc{S} (\R)$, 
\[\<H,(-\Scal) H\> = \frac{c_++c_-}{4}\iint_{\R^2} \frac{1}{|v|^{1+\alpha}} \big(H (u+v) - H (u) \big)^2 \,du\, dv\]
\end{remark}

\begin{remark}
As stated in Remark \ref{rmk:L}, for $H \in \mc{S} (\R)$, $\mathcal{Y}_t (\Lcal H)$ may not be well defined since $\Lcal H$ may not belong to $\mc{S} (\R)$. However, by continuity of the operator $\Lcal: \mc{S} (\R) \rightarrow L^2 (\R)$, one could first approximate $\Lcal H$ in $L^2 (\R)$ by some function $(\Lcal H)_\varepsilon \in \mc{S} (\R)$, and then define $\mathcal{Y}_t (\Lcal H)$ as the limit of $\mathcal{Y}_t \big((\Lcal H)_\varepsilon\big)$ by using the stationary of the process $\mc{Y}_t$. We refer the readers to \cite[Subsection 2.3]{gonccalves2018density} for detailed implement of the above approximation procedure.
\end{remark}

\begin{remark}
	The stationary solutions to \eqref{ou1} and \eqref{ou2} are unique in the sense of distributions. We refer the readers to \cite[Proposition C.1]{gonccalves2018density} for the proof of the uniqueness in the case $n=1$, which could be easily extended to $n \geq 2$.
\end{remark}

\subsection{The coupled  fractional Burgers equation}\label{subsec:fbe}  For $\varepsilon > 0$, define $\iota: \R \rightarrow \R_+$ as
\begin{equation}\label{iota}
\iota_\varepsilon (u) = \frac{1}{2 \varepsilon} \mathbf{1}_{[-\varepsilon,\varepsilon]} (u), \quad u \in \R.
\end{equation}
Let $G_\varepsilon \in \mc{S} (\R)$ be an approximation of the function $\iota (\cdot)$ such that
\begin{equation}
	\lim_{\varepsilon \rightarrow 0} \varepsilon^{-1} \|G_\varepsilon - \iota_\varepsilon\|_{L^2 (\R)}^2 = 0, \qquad \|G_\varepsilon\|_{L^2 (\R)}^2 \leq 2 \|\iota_\varepsilon\|_{L^2 (\R)}^2 = \varepsilon^{-1}.
\end{equation}
An example is to define $G_\varepsilon$ as the convolution of $\iota (\cdot)$ with smooth kernels.

Let $\{\mc{Y}_t, 0 \leq t \leq T\}$ be an $\mc{S}' (\R)^n$--valued stationary process. For $H \in \mc{S} (\R)$, for $1 \leq i \leq n$, define
\begin{equation}\label{a_epsilon}
\mc{A}^{i,\varepsilon}_t (H) = \sum_{j,k=1}^n \big(\partial_{\rho^{j}} \partial_{\rho^{k}} \tilde{g}_i\big) (\rhobf)  \int_0^t \int_{\R} \tau_{-u} \mc{Y}_s^j (G_\varepsilon) \tau_{-u} \mc{Y}_s^k (G_\varepsilon)  H^\prime (u)\,du\,ds,
\end{equation}
where the shift operators $\tau_u, u \in \R$, are defined as
\[\tau_u \mc{Y}^i_t (H) = \mc{Y}^i_t (\tau_u H), \quad \tau_u H (v) = H(u+v)\]
for $\mc{Y}^i_t \in \mc{S}^\prime (\R), H \in \mc{S} (\R)$. To abuse notations, we also define $\tau_x$, $x \in \Z$, which act  $f: \Omega \rightarrow \R$ and $\etabf \in \Omega$ as
\[\tau_x f (\etabf) = f (\tau_x \etabf), \quad \tau \etabf (y) = \etabf(x+y), \;y \in \Z.\]
We say the process $\{\mc{Y}_t, 0 \leq t \leq T\}$  satisfies the {\it $L^2$ energy condition} if for any $H \in  \mc{S} (\R)$ and for any $1 \leq i \leq n$, the sequence $\{\mc{A}^{i,\varepsilon}_t (H) , 0 \leq t \leq T\}_{\varepsilon > 0}$ is a uniform $L^2$ Cauchy sequence in the following sense,
\begin{equation}\label{energy_cdondition}
\lim_{\varepsilon \rightarrow 0} \sup_{\delta < \varepsilon} \E \Big[ \sup_{0 \leq t \leq T} \big(\mc{A}^{i,\varepsilon}_t (H) - \mc{A}^{i,\delta}_t (H)\big)^2\Big] = 0,
\end{equation} 
and the limit does not depend on the choice of the approximating function $\{G_\varepsilon\}$.  As a result,  for any $H \in  \mc{S} (\R)$ and for any $1 \leq i \leq n$, the limit
\[\mc{A}^{i}_t (H):=\lim_{\varepsilon \rightarrow 0} \mc{A}^{i,\varepsilon}_t (H)\]
exists in the uniform $L^2$ norm. Moreover, the process $(\mc{A}^{i}_\cdot)_{1 \leq i \leq n} \in C([0,T],\mc{S}' (\R)^n)$.

\begin{definition}\label{def:fburgers}
{\rm (The case $\alpha = 3/2$).} We say $\{\mc{Y}_t, 0 \leq t \leq T\}$ is a stationary energy solution of the coupled stochastic fractional Burgers equation,  for any $1 \leq i \leq n$,
\begin{equation}\label{fbe}
\partial_t \mc{Y}_t^i =   (\partial_{\rho^{i}} \tilde{g}_i) (\rhobf) \Lcal^* \mc{Y}^i_t + \frac{m}{2} \sum_{j,k=1}^n \big(\partial_{\rho^{j}} \partial_{\rho^{k}} \tilde{g}_i\big) (\rhobf) \nabla ( \mc{Y}^j_t \mc{Y}^k_t) +2  \sqrt{\tilde{g}_i (\rhobf) (-\Scal)} \dot{\mc{W}}^i_t, 
\end{equation}
if the following conditions are satisfied:
\begin{itemize}
	\item the process  $\{\mc{Y}_t, 0 \leq t \leq T\}$ is stationary;
	\item the process $\{\mc{Y}_t, 0 \leq t \leq T\}$ satisfies the $L^2$-energy condition. As s result, for any $1 \leq i \leq n$, there exists a process $\mc{A}^i_\cdot \in C([0,T],\mc{S} (\R))$ such that for any $H \in \mc{S} (\R)$, $\mc{A}^i_\cdot (H)$ is the uniform $L^2$ limit of $\mc{A}^{i,\varepsilon}_\cdot (H)$ as $\varepsilon \rightarrow 0$;
	\item  for any $1 \leq i \leq n$ and for any trajectory $H: [0,T] \rightarrow \mc{S} (\R)$,
	\begin{equation}
		\mc{M}_t^{i} (H) = \mc{Y}_t^{i} (H_t) - \mc{Y}_0^{i} (H_0)  - (\partial_{\rho^{i}} \tilde{g}_i) (\rhobf) \int_0^t  \mc{Y}_s^{i} ((\partial_s + \Lcal) H_s) \,ds - \frac{m}{2} \mc{A}^i_t (H) 
	\end{equation}
is a continuous martingale with quadratic variation 
\[4 \tilde{g}_i (\rhobf) \int_0^t \<H_s,(-\Scal) H_s\> ds.\]
Moreover, for $1 \leq i \neq j \leq n$, the martingales $\mc{M}_\cdot^{i}$ and $\mc{M}_\cdot^{j}$ are independent.
\end{itemize}
\end{definition}

\subsection{Stationary fluctuations}\label{subsec:sf} Let $\{\etabf_t\}$ be the process with generator $N^\alpha \gen$ and with initial distribution $\nu_{\rhobf}$.  Denote by $\P_{\nu_{\rhobf}}$ the probability measure on the path space $D([0,T],\Omega)$ associated to the process $\{\etabf_t\}$ and by $\E_{\nu_{\rhobf}}$ the corresponding expectation. We suppress the dependence of the notations on $N$ for short. 

In the case $1 \leq \alpha \leq 3/2$, we need in addition the following frame condition so that the fluctuation fields  defined below are viewed in the same reference frame.  

\begin{assumption}[Frame condition]\label{assump:frame}
	Assume the density $\rhobf$ satisfies that for any $1 \leq i \leq n$,
	\[\lambda := (\partial_{\rho^{i}} \tilde{g}_i) (\rhobf),  \quad (\partial_{\rho^{j}} \tilde{g}_i) (\rhobf) = 0,\;\forall j \neq i.\]
\end{assumption}

Define
\[m^\alpha_N = \begin{cases}
	0, \quad &\alpha < 1;\\
	N \sum_{|x| \leq N} x p(x), \quad &\alpha =1;\\
	N^\alpha m, \quad &\alpha >1.
\end{cases}\]
The density fluctuation fields of the process are defined as $\mathcal{Y}^N_t = \big(\mathcal{Y}^{i,N}_t\big)_{1 \leq i \leq n} \in \mc{S}^\prime (\R)^n$, where for any $H \in \mc{S} (\R)$, 
\[ \mathcal{Y}^{i,N}_t (H) = \frac{1}{\sqrt{N}} \sum_{x \in \Z} \bar{\eta}^i_t (x) H_{\lambda,t} \big(\tfrac{x}{N}\big).\]
Above, denote $H_{\lambda,t} (\cdot) = H \big(\cdot-\tfrac{\lambda t m^\alpha_N}{N}\big)$ and $\bar{\eta}^i_t (x) = \eta^i_t (x)- \rho^i$. Note that in the case $0 < \alpha < 1$, we have $H_{\lambda,t} = H$.

Throughout the  article, we need Assumptions \ref{assum:stationary}, \ref{assump:inv} and \ref{assump:sg} if $0 < \alpha < 1$, and in addition Assumption \ref{assump:frame} if $1 \leq \alpha \leq 3/2$.

Now we are ready to state the main result of the article. 

\begin{theorem}\label{thm:main}
Let the process start from $\nu_{\rhobf}$.

{\rm (1)} If $0 < \alpha < 1$, then the sequence of processes $\{\mathcal{Y}^{N}_t, 0 \leq t \leq T\}_{N \geq 1}$ converges in distribution, as $N \rightarrow \infty$, with respect to the Skorohod topology of $D([0,T], \mc{S}^\prime (\R)^n)$ to the stationary energy solution of the Ornstein-Uhlenbeck process \eqref{ou1}.

{\rm (2)} If $1 \leq  \alpha < 3/2$, then the sequence of processes $\{\mathcal{Y}^{N}_t, 0 \leq t \leq T\}_{N \geq 1}$ converges in the sense as in {\rm (1)} to the stationary energy solution of the Ornstein-Uhlenbeck process \eqref{ou2}.

{\rm (3)} If $\alpha=3/2$, then  the sequence of processes $\{\mathcal{Y}^{N}_t, 0 \leq t \leq T\}_{N \geq 1}$ is tight with respect to the Skorohod topology of $D([0,T], \mc{S}^\prime (\R)^n)$. Moreover, any of its limit points is a stationary energy solution to the fractional Burgers equation \eqref{fbe}.
\end{theorem}

\section{The Boltzmann-Gibbs Principle}\label{sec:bgp}

For $\etabf \in \Omega$ and integers $\ell \geq 1$, let $\etabf^{(\ell)} (x)= \big(\eta^{i,(\ell)} (x) \big)_{1 \leq i \leq n}$ be the spatial average of $\etabf$ in a box of length $2\ell+1$ centered at site $x$, i.e., 
\[\eta^{i,(\ell)} (x) = \frac{1}{2\ell+1} \sum_{|y-x| \leq \ell} \eta^{i} (y).\]
Similarly, define
\[\bar{\eta}^{i,(\ell)} (x) = \frac{1}{2\ell+1} \sum_{|y-x| \leq \ell} \bar{\eta}^{i} (y).\]

The aim of this section is to prove the following estimate.

\begin{proposition}[Boltzmann-Gibbs Principle]\label{prop:bg}
	For any $H: \R \rightarrow \R$ and any local function $f: \Omega \rightarrow \R$ supported on sites in $\Lambda_{\ell_0}$ for some $\ell_0 > 0$, there exists some constant $C=C(\rhobf,\ell_0,\alpha)$  such that if $\tilde{f} (\rhobf) = 0$, then
	\begin{multline}\label{bg1}
	 \E_{\nu_{\rhobf}} \Big[\sup_{0 \leq t \leq T} \Big( \int_0^t \sum_{x \in \Z} \Big[\tau_x f (\etabf_s)  - \sum_{j=1}^n (\partial_{\rho^{j}} \tilde{f}) (\rhobf) \bar{\eta}^{j,(\ell)}_s (x) \Big] H_{\lambda,s} \big(\tfrac{x}{N}\big) ds \Big)^2 \Big] \\
	 \leq C  \|f\|_{L^5 (\nu_{\rhobf})}^2 \Big\{ \frac{T \ell^\alpha}{N^{\alpha-1}} \Big( \frac{1}{N} \sum_{x \in \Z} H \big(\tfrac{x}{N}\big)^2 \Big)  + \frac{T^2 N^2}{\ell^2}\Big( \frac{1}{N} \sum_{x \in \Z} |H \big(\tfrac{x}{N}\big)| \Big)^2 \Big\}.
	\end{multline}	 
If  $\tilde{f} (\rhobf) = \nabla \tilde{f} (\rhobf) = 0$, then 
	\begin{multline}\label{bg2}
	\E_{\nu_{\rhobf}} \Big[\sup_{0 \leq t \leq T} \Big( \int_0^t \sum_{x \in \Z} \Big[\tau_x f (\etabf_s)  \\
	- \frac{1}{2} \sum_{j,k=1}^n (\partial_{\rho^{j}} \partial_{\rho^{k}}  \tilde{f}) (\rhobf) \big\{\bar{\eta}^{j,(\ell)}_s (x)  \bar{\eta}^{k,(\ell)}_s (x) - (2\ell+1)^{-1} \Gamma_{j,k} (\rhobf) \big\}\Big] H_{\lambda,s} \big(\tfrac{x}{N}\big) ds \Big)^2 \Big] \\
	\leq C  \|f\|_{L^5 (\nu_{\rhobf})}^2 \Big\{ \frac{T\beta_\alpha (\ell)}{N^{\alpha-1}}  \Big( \frac{1}{N} \sum_{x \in \Z} H \big(\tfrac{x}{N}\big)^2 \Big) + \frac{T^2N^2}{\ell^3} \Big( \frac{1}{N} \sum_{x \in \Z} |H \big(\tfrac{x}{N}\big)| \Big)^2 \Big\},
\end{multline}
where
\[\beta_\alpha (\ell) = \begin{cases}
	\ell^{\alpha-1}, \quad &\alpha > 1;\\
	(\log \ell)^2, \quad &\alpha = 1;\\
	1,\quad &0 < \alpha < 1.
\end{cases}\]
Recall $\Gamma_{j,k} (\rhobf)$ is the covariance of $\eta^j (0)$ and $\eta^k (0)$ with respect to $\nu_{\rhobf}$.
\end{proposition}

We adapt the proof of \cite[Theorem 4.1]{bernardin2021derivation} to the long-jump setting, which  is divided into several lemmas. Along the proof, the constant $C$ may be different from line to line.

For any local, bounded function $f: \Omega \rightarrow \R$, denote $D (f;\nu_{\rhobf})$, respectively $D_{\ell} (f;\nu_{\rhobf})$, the grand Dirichlet form, respectively the grand Dirichlet form restricted to the interval $\Lambda_\ell$,
\begin{align*}
	D (f;\nu_{\rhobf}) &= \<f(-\gen)f\>_{\nu_{\rhobf}} = \frac{1}{2}\sum_{x,y \in \Z} \sum_{i=1}^n  s(y-x) E_{\nu_{\rhobf}} \Big[ g_i (\etabf (x)) \big\{f (\etabf^{i;x,y}) - f (\etabf)\big\}^2 \Big],\\
	D_{\ell} (f;\nu_{\rhobf}) &= \frac{1}{2}\sum_{x,y \in \Lambda_\ell} \sum_{i=1}^n  s(y-x) E_{\nu_{\rhobf}} \Big[ g_i (\etabf (x)) \big\{f (\etabf^{i;x,y}) - f (\etabf)\big\}^2 \Big].
\end{align*}
One could check  that
\begin{equation}\label{3.0}
	\sum_{x \in \Z} D_{\ell} (\tau_x f;\nu_{\rhobf}) \leq (2\ell+1) D (f;\nu_{\rhobf}).
\end{equation}
Indeed, we write the left side in \eqref{3.0} as
\begin{multline*}
\frac{1}{2}\sum_{x \in \Z}  \sum_{y,z \in \Lambda_\ell} \sum_{i=1}^n  s(y-z) E_{\nu_{\rhobf}} \Big[ g_i (\etabf (y)) \big\{(\tau_x f) (\etabf^{i;y,z}) - (\tau_x f )(\etabf)\big\}^2 \Big]	\\
= \frac{1}{2}\sum_{x \in \Z}  \sum_{y,z \in \Lambda_\ell} \sum_{i=1}^n  s(y-z) E_{\nu_{\rhobf}} \Big[ g_i ( (\tau_{-x} \etabf) (y)) \big\{(\tau_x f) ( (\tau_{-x} \etabf)^{i;y,z}) - ( \tau_x f )( \tau_{-x} \etabf)\big\}^2 \Big]\\
= \frac{1}{2}\sum_{x \in \Z}  \sum_{y,z \in \Lambda_\ell} \sum_{i=1}^n  s(y-z) E_{\nu_{\rhobf}} \Big[ g_i ( \etabf (y-x)) \big\{ f (  \etabf^{i;y-x,z-x}) -  f (  \etabf)\big\}^2 \Big].
\end{multline*}
In the first identity, we make the transformation $\etabf \mapsto \tau_{-x} \etabf$, and in the second one we use
\[\tau_x (\tau_{-x} \etabf)^{i;y,z} = \etabf^{i,y-x,z-x}.\]
First summing over $x \in \Z, y \in \Lambda_\ell$, then over $z \in \Lambda_\ell$, we prove \eqref{3.0}. 

\begin{lemma}\label{lem:3.2}
Let $r: \Omega \rightarrow \R$ be a local $L^4 (\nu_{\rhobf})$ function. Suppose the support of $r$ is contained on sites in $\Lambda_{\ell_0}$ for some $\ell_0 > 0$, and that $E_{\nu_{\rhobf}} \big[r(\etabf)|\etabf^{(\ell_0)} (0)\big] = 0$ almost surely.  Then, for any local, bounded function $f: \Omega \rightarrow \R$,
\[\big|E_{\nu_{\rhobf}} \big[r(\etabf)f(\etabf)\big]\big| \leq E_{\nu_{\rhobf}} \Big[W\Big(\sum_{x \in \Lambda_{\ell_0}} \etabf (x),\ell_0\Big)^2\Big]^{1/4} \|r\|_{L^4 (\nu_{\rhobf})} D_{\ell_0} (f;\nu_{\rhobf})^{1/2}.\]
In particular, under the Assumption \ref{assump:sg},  there exists some constant $C=C(\rhobf)$ such that
\[\big|E_{\nu_{\rhobf}} \big[r(\etabf)f (\etabf) \big]\big| \leq C \ell_0^{\alpha/2} \|r\|_{L^4 (\nu_{\rhobf})} D_{\ell_0} (f;\nu_{\rhobf})^{1/2}.\]
\end{lemma}

The proof of the above lemma uses the definition of the spectral gap and is exactly the same as \cite[Proposition 5.2]{bernardin2021derivation}. For this reason, we omit it.

\begin{lemma}\label{lem:3.3}
	For any function $H: \R \rightarrow \R$ and any local $L^4 (\nu_{\rhobf})$ function $f: \Omega \rightarrow \R$ such that $\tilde{f} (\rhobf) = 0$, there exists some constant $C = C(\rhobf)$ such that
	\begin{multline}\label{3.1}
		\E_{\nu_{\rhobf}} \Big[ \sup_{0 \leq t \leq T} \Big( \int_0^t \sum_{x \in \Z} \Big[\tau_x f (\etabf_s)  - \E_{\nu_{\rhobf}} \big[\tau_x f (\etabf_s)| \etabf_s^{(\ell )}(x)   \big]\Big] H_{\lambda,s} \big(\tfrac{x}{N}\big) ds \Big)^2 \Big] \\
		 \leq C  \frac{T \ell^{1+\alpha}}{N^{\alpha-1}} \|f\|_{L^4 (\nu_{\rhobf})}^2 \Big(\frac{1}{N}\sum_{x \in \Z} H \big(\tfrac{x}{N}\big)^2\Big).
	\end{multline}	 
\end{lemma}

\begin{proof}
	By Kipnis--Varadhan inequality \cite{klscaling}, we bound the expectation in \eqref{3.1} by 
	\begin{equation}\label{3.2}
		20 \int_0^T \Big\|\sum_{x \in \Z} \Big[\tau_x f (\etabf)  - E_{\nu_{\rhobf}} \big[\tau_x f (\etabf)| \etabf^{(\ell)} (x)   \big]\Big] H_{\lambda,s} \big(\tfrac{x}{N}\big) \Big\|_{-1,N}^2 ds,
	\end{equation}
where the $H_{-1}$ norm of function $f: \Omega \rightarrow \R$ is defined as
\begin{align*}
	\|f\|_{-1,N} &= \sup \Big\{ \frac{E_{\nu_{\rhobf}}[f\phi]}{\|\phi\|_{1,N}}: \phi \; \text{bounded, local}, \; \|\phi\|_{1,N} > 0\Big\},\\
	\|\phi\|_{1,N}^2 &= N^\alpha 	D (\phi;\nu_{\rhobf}).
\end{align*}
By translation invariance of the measure $\nu_{\rhobf}$, we rewrite \eqref{3.2} as
\[20 T \Big\| \sum_{x \in \Z} \Big[ \tau_x f (\etabf)  - E_{\nu_{\rhobf}} \big[\tau_x f (\etabf)| \etabf^{(\ell)} (x)  \big]\Big] H \big(\tfrac{x}{N}\big) \Big\|_{-1,N}^2.\]
Note that
\begin{multline*}
\Big\| \sum_{x \in \Z} \Big[ \tau_x f (\etabf)  - E_{\nu_{\rhobf}} \big[\tau_x f (\etabf)| \etabf^{(\ell)} (x)  \big]\Big] H \big(\tfrac{x}{N}\big) \Big\|_{-1,N} 
\\= N^{-\alpha/2} \sup_{\phi} \Big\{ D (\phi;\nu_{\rhobf})^{-1/2}  \sum_{x \in \Z} E_{\nu_{\rhobf}} \Big[\Big( \tau_x f (\etabf)  - E_{\nu_{\rhobf}} \big[\tau_x f (\etabf)| \etabf^{(\ell)} (x)   \big]\Big) \phi (\etabf) \Big] H \big(\tfrac{x}{N}\big)  \Big]\Big\} \\
= N^{-\alpha/2} \sup_{\phi} \Big\{ D (\phi;\nu_{\rhobf})^{-1/2}  \sum_{x \in \Z} E_{\nu_{\rhobf}} \Big[\Big(  f (\etabf)  - E_{\nu_{\rhobf}} \big[f (\etabf)| \etabf^{(\ell)}  (0)   \big]\Big) (\tau_{-x} \phi) (\etabf)  \Big] H \big(\tfrac{x}{N}\big)  \Big]\Big\}.
\end{multline*}
By Lemma \ref{lem:3.2}, Young's inequality and \eqref{3.0}, we bound the last line by, for any $A > 0$,
\begin{multline*}
C(\rhobf) \ell^{\alpha/2} N^{-\alpha/2} \sup_{\phi} \Big\{ D (\phi;\nu_{\rhobf})^{-1/2}  \sum_{x \in \Z} \|f\|_{L^4(\nu_{\rhobf})} D_\ell(\tau_{-x} \phi; \nu_{\rhobf})^{1/2} \big| H \big(\tfrac{x}{N}\big) \big|  \Big]\Big\}\\
\leq  C(\rhobf) \ell^{\alpha/2} N^{-\alpha/2} \sup_{\phi} \Big\{ D (\phi;\nu_{\rhobf})^{-1/2} \Big( A \sum_{x \in \Z} \|f\|^2_{L^4(\nu_{\rhobf})}  H \big(\tfrac{x}{N}\big)^2 + A^{-1} \sum_{x \in \Z} D_\ell(\tau_{-x} \phi; \nu_{\rhobf})  \Big)\Big\} \\
\leq  C(\rhobf) \ell^{\alpha/2} N^{-\alpha/2} \sup_{\phi} \Big\{ D (\phi;\nu_{\rhobf})^{-1/2} \Big( A \sum_{x \in \Z} \|f\|^2_{L^4(\nu_{\rhobf})}  H \big(\tfrac{x}{N}\big)^2 + A^{-1} \ell  D ( \phi; \nu_{\rhobf})  \Big)\Big\}.
\end{multline*} 
In the last inequality, we use \eqref{3.0}. By taking 
\[A = \ell^{1/2}  \|f\|^{-1}_{L^4(\nu_{\rhobf})} \Big[\sum_{x \in \Z}  H \big(\tfrac{x}{N}\big)^2\Big]^{-1/2}  D ( \phi; \nu_{\rhobf})^{1/2},\]
we obtain the desired bound.
\end{proof}

The next lemma doubles the length of the box, which is useful when applying the  multi-scale analysis. 

\begin{lemma}\label{lem:3.4}
	For any function $H: \R\rightarrow \R$ and any local $L^5 (\nu_{\rhobf})$ function $f: \Omega \rightarrow \R$ such that $\tilde{f} (\rhobf) = 0$, there exists some constant $C=C(\rhobf)$ such that for large enough $\ell$,  
	\begin{multline*}
		\E_{\nu_{\rhobf}} \Big[ \sup_{0 \leq t \leq T} \Big( \int_0^t \sum_{x \in \Z} \Big[\E_{\nu_{\rhobf}} \big[\tau_x f (\etabf_s)| \etabf^{(\ell)}_s (x)   \big]  - \E_{\nu_{\rhobf}} \big[\tau_x f (\etabf_s)| \etabf^{(2 \ell)}_s (x)   \big]\Big] H_{\lambda,s} \big(\tfrac{x}{N}\big) ds \Big)^2 \Big] \\
		\leq C  \frac{T \ell^{\alpha}}{N^{\alpha-1}} \|f\|_{L^5 (\nu_{\rhobf})}^2 \Big(\frac{1}{N} \sum_{x \in \Z} H \big(\tfrac{x}{N}\big)^2\Big).
	\end{multline*}	 
If in addition $\nabla \tilde{f} (\rhobf) = 0$, then 
	\begin{multline*}
	\E_{\nu_{\rhobf}} \Big[ \sup_{0 \leq t \leq T} \Big( \int_0^t \sum_{x \in \Z} \Big[\E_{\nu_{\rhobf}} \big[\tau_x f (\etabf_s)| \etabf^{(\ell)}_s (x)   \big]  - \E_{\nu_{\rhobf}} \big[\tau_x f (\etabf_s)| \etabf^{(2 \ell)}_s (x)   \big]\Big] H_{\lambda,s} \big(\tfrac{x}{N}\big) ds \Big)^2 \Big] \\
	\leq C \frac{T \ell^{\alpha-1}}{N^{\alpha-1}} \|f\|_{L^5 (\nu_{\rhobf})}^2 \Big(\frac{1}{N} \sum_{x \in \Z} H \big(\tfrac{x}{N}\big)^2\Big).
\end{multline*}	
\end{lemma}

\begin{proof}
First observe that
\[\E_{\nu_{\rhobf}} \Big[ \E_{\nu_{\rhobf}} \big[\tau_x f (\etabf_s)| \etabf^{(\ell)}_s (x)   \big] \Big| \etabf^{(2\ell)}_s (x) \Big] = \E_{\nu_{\rhobf}} \big[\tau_x f (\etabf_s)| \etabf^{(2 \ell)}_s (x)   \big]\]	
since the support of $f$ is contained on sites in $\Lambda_\ell$ for $\ell$ large enough. Then, following the proof of  Lemma \ref{lem:3.3}, the expectation in the lemma are bounded by
\[C  \frac{T \ell^{\alpha+1}}{N^{\alpha-1}} \big\|E_{\nu_{\rhobf}} \big[ f (\etabf)| \etabf^{(\ell)} (0)   \big]  - E_{\nu_{\rhobf}} \big[ f (\etabf )| \etabf^{(2 \ell)} (0)\big] \big\|_{L^4 (\nu_{\rhobf})}^2 \Big(\frac{1}{N} \sum_{x \in \Z} H \big(\tfrac{x}{N}\big)^2\Big).\]
The equivalence of ensembles \cite[Theorem 5.1]{bernardin2021derivation} states that if $\tilde{f} (\rhobf) = 0$, then 
\begin{equation}\label{equiofensem1}
\big\| E_{\nu_{\rhobf}} \big[ f (\etabf)| \etabf^{(\ell)} (0)   \big]  - \sum_{j=1}^n (\partial_{\rho^{j}} \tilde{f}) (\rhobf) \bar{\eta}^{j,(\ell)} (0)   \big\|_{L^4 (\nu_{\rhobf})}^2 \leq C(\rhobf) \ell^{-2} \|f\|_{L^5 (\nu_{\rhobf})}^2,
\end{equation}
and that if $\tilde{f} (\rhobf) = \nabla \tilde{f} (\rhobf) = 0$, then 
\begin{equation}\label{equiofensem2}
\big\| E_{\nu_{\rhobf}} \big[ f (\etabf)| \etabf^{(\ell)} (0)   \big]  - \sum_{j,k=1}^n (\partial_{\rho^{j}} \partial_{\rho^{k}}\tilde{f}) (\rhobf) \big\{ \bar{\eta}^{j,(\ell)} (0)  \bar{\eta}^{k,(\ell)} (0) - \tfrac{\Gamma_{j,k} (\rhobf)}{2\ell+1}\big\} \big\|_{L^4 (\nu_{\rhobf})}^2 \leq C(\rhobf) \ell^{-3} \|f\|_{L^5 (\nu_{\rhobf})}^2.
\end{equation}
Therefore, by  Cauchy-Schwarz inequality, if $\tilde{f} (\rhobf) = 0$, then 
\begin{multline*}
\Big\| E_{\nu_{\rhobf}} \big[ f (\etabf)| \etabf^{(\ell)} (0)   \big]  - E_{\nu_{\rhobf}} \big[ f (\etabf )| \etabf^{(2 \ell)} (0)   \big] \Big\|_{L^4 (\nu_{\rhobf})}^2 
\leq 3  \Big\| E_{\nu_{\rhobf}} \big[ f (\etabf)| \etabf^{(\ell)} (0)   \big]  - \sum_{j=1}^n (\partial_{\rho^{j}} \tilde{f}) (\rhobf) \bar{\eta}^{j,(\ell)} (0)  \Big\|_{L^4 (\nu_{\rhobf})}^2 \\
+ 3 \Big\| \sum_{j=1}^n (\partial_{\rho^{j}} \tilde{f}) (\rhobf) \bar{\eta}^{j,(\ell)} (0)   - \sum_{j=1}^n (\partial_{\rho^{j}} \tilde{f}) (\rhobf) \bar{\eta}^{j,(2\ell)} (0)    \big] \Big\|_{L^4 (\nu_{\rhobf})}^2\\
+ 3 \Big\| \sum_{j=1}^n (\partial_{\rho^{j}} \tilde{f}) (\rhobf) \bar{\eta}^{j,(2\ell)} (0)   - E_{\nu_{\rhobf}} \big[ f (\etabf )| \etabf^{(2 \ell)} (0)   \big] \Big\|_{L^4 (\nu_{\rhobf})}^2 \\
\leq C(\rhobf) \big\{  \ell^{-2} \|f\|_{L^5 (\nu_{\rhobf})}^2 + \ell^{-1} \|f\|_{L^2 (\nu_{\rhobf})}^2 \big\} \leq C(\rhobf) \ell^{-1} \|f\|_{L^5 (\nu_{\rhobf})}^2.
\end{multline*}
In the penultimate inequality, we use \eqref{equiofensem1} and the fact that
\[\|\bar{\eta}^{j,(\ell)} (0) \|_{L^4 (\nu_{\rhobf})}^4 \leq C(\rhobf) \ell^{-2}, \quad |(\partial_{\rho^{j}} \tilde{f}) (\rhobf) | \leq C(\rhobf) \|f\|_{L^2 (\nu_{\rhobf})}. \]
If $\tilde{f} (\rhobf) = \nabla \tilde{f} (\rhobf) = 0$, by using \eqref{equiofensem2} and
\[ \big\| \bar{\eta}^{j,(\ell)} (0)  \bar{\eta}^{k,(\ell)} (0) - \tfrac{\Gamma_{j,k} (\rhobf)}{2\ell+1} \big\|_{L^4 (\nu_{\rhobf})}^4 \leq C(\rhobf) \ell^{-4}, \quad |(\partial_{\rho^{j}} \partial_{\rho^{k}}\tilde{f}) (\rhobf) | \leq C(\rhobf) \|f\|_{L^2 (\nu_{\rhobf})}, \] 
we bound
\[\Big\| E_{\nu_{\rhobf}} \big[ f (\etabf)| \etabf^{(\ell)} (0)   \big]  - E_{\nu_{\rhobf}} \big[ f (\etabf )| \etabf^{(2 \ell)} (0)   \big] \Big\|_{L^4 (\nu_{\rhobf})}^2 \leq C(\rhobf) \ell^{-2} \|f\|_{L^5 (\nu_{\rhobf})}^2,\]
which is sufficient to conclude the proof. 
\end{proof}

\begin{lemma}\label{lem:3.5}
	Let $H: \R \rightarrow \R$ be some function.  For any local $L^5 (\nu_{\rhobf})$ function $f: \Omega \rightarrow \R$ supported on sites in $\Lambda_{\ell_0}$ such that $\tilde{f} (\rhobf) = 0$, there exists some constant $C = C(\rhobf, \ell_0,\alpha)$ such that for  $\ell \geq \ell_0$,  
	\begin{multline*}
		\E_{\nu_{\rhobf}} \Big[ \sup_{0 \leq t \leq T} \Big( \int_0^t \sum_{x \in \Z} \Big[\E_{\nu_{\rhobf}} \big[\tau_x f (\etabf_s)| \etabf^{(\ell_0)}_s (x)   \big]  - \E_{\nu_{\rhobf}} \big[\tau_x f (\etabf_s)| \etabf^{(\ell)}_s (x)   \big]\Big] H_{\lambda,s} \big(\tfrac{x}{N}\big) ds \Big)^2 \Big] \\
		\leq C \frac{T \ell^{\alpha}}{N^{\alpha-1}} \|f\|_{L^5 (\nu_{\rhobf})}^2 \Big(\frac{1}{N}\sum_{x \in \Z} H \big(\tfrac{x}{N}\big)^2\Big).
	\end{multline*}	 
	If in addition $\nabla \tilde{f} (\rhobf) = 0$, then 
	\begin{multline*}
		\E_{\nu_{\rhobf}} \Big[ \sup_{0 \leq t \leq T} \Big( \int_0^t \sum_{x \in \Z} \Big[\E_{\nu_{\rhobf}} \big[\tau_x f (\etabf_s)| \etabf^{(\ell_0)}_s (x)   \big]  - \E_{\nu_{\rhobf}} \big[\tau_x f (\etabf_s)| \etabf^{(\ell)}_s (x)   \big]\Big] H_{\lambda,s} \big(\tfrac{x}{N}\big) ds \Big)^2 \Big] \\
		\leq C \frac{T \beta_\alpha (\ell)}{N^{\alpha-1}} \|f\|_{L^5 (\nu_{\rhobf})}^2 \Big(\frac{1}{N}\sum_{x \in \Z} H \big(\tfrac{x}{N}\big)^2\Big).
	\end{multline*}	
Recall the definition of $\beta_\alpha (\cdot)$ from Proposition \ref{prop:bg}.
\end{lemma}

\begin{proof}
For any $\ell \geq \ell_0$, denote $\ell = 2^{m+1} \ell_0 + r$ for some positive  integer $m$ and $0 \leq r \leq 2^{m+1} \ell_0 -1$. Then,  we split
\begin{multline*}
	\E_{\nu_{\rhobf}} \big[\tau_x f (\etabf_s)| \etabf^{(\ell)}_s (x)   \big]  - \E_{\nu_{\rhobf}} \big[\tau_x f (\etabf_s)| \etabf^{(\ell_0)}_s (x)   \big] \\
	= \sum_{j=0}^m \Big\{ \E_{\nu_{\rhobf}} \big[\tau_x f (\etabf_s)|  \etabf^{(2^{j+1}\ell_0)}_s (x)   \big]  - \E_{\nu_{\rhobf}} \big[\tau_x f (\etabf_s)|  \etabf^{(2^{j}\ell_0)}_s (x)   \big]\Big\}\\
	+  \E_{\nu_{\rhobf}} \big[\tau_x f (\etabf_s)| \etabf_s^{(\ell)} (x)   \big]  - \E_{\nu_{\rhobf}} \big[\tau_x f (\etabf_s)| \etabf_s^{ (2^{m+1} \ell_0)} (x)   \big].
\end{multline*}
By Cauchy-Schwarz inequality, we bound the expectation in the lemma by 
\begin{multline*}
\Big\{  \sum_{j=0}^m 	\E_{\nu_{\rhobf}} \Big[ \sup_{0 \leq t \leq T} \Big( \int_0^t \sum_{x \in \Z} \Big[ \E_{\nu_{\rhobf}} \big[\tau_x f (\etabf_s)|  \etabf^{(2^{j+1}\ell_0)}_s (x)   \big]  - \E_{\nu_{\rhobf}} \big[\tau_x f (\etabf_s)|  \etabf^{(2^{j}\ell_0)}_s (x)   \big]\Big] H_{\lambda,s} \big(\tfrac{x}{N}\big) ds \Big)^2 \Big]^{1/2} \\
+	\E_{\nu_{\rhobf}} \Big[ \sup_{0 \leq t \leq T} \Big( \int_0^t \sum_{x \in \Z} \Big[\E_{\nu_{\rhobf}} \big[\tau_x f (\etabf_s)| \etabf_s^{(\ell)} (x)   \big]  - \E_{\nu_{\rhobf}} \big[\tau_x f (\etabf_s)| \etabf_s^{ (2^{m+1} \ell_0)}  (x)   \big]\Big] H_{\lambda,s} \big(\tfrac{x}{N}\big) ds \Big)^2 \Big]^{1/2}     \Big\}^2.
\end{multline*}
By Lemma \ref{lem:3.4}, if $\tilde{f} (\rhobf) = 0$,  the last expression is bounded by
\[\frac{C(\rhobf)T\ell_0^\alpha}{N^{\alpha-1}}   \|f\|_{L^5 (\nu_{\rhobf})}^2 \Big( \frac{1}{N} \sum_{x \in \Z} H \big(\tfrac{x}{N}\big)^2\Big) \Big\{ 2^{(m+1)\alpha/2} + \sum_{j=0}^m 2^{j \alpha/2} \Big\}^2 \leq \frac{C(\rhobf,\alpha)T\ell^\alpha}{N^{\alpha-1}}   \|f\|_{L^5 (\nu_{\rhobf})}^2 \Big( \frac{1}{N} \sum_{x \in \Z} H \big(\tfrac{x}{N}\big)^2\Big) .\]
If $\tilde{f} (\rhobf) = \nabla \tilde{f} (\rhobf) =0$, the bound is 
\[\frac{C(\rhobf)T\ell_0^{\alpha-1}}{N^{\alpha-1}}   \|f\|_{L^5 (\nu_{\rhobf})}^2 \Big( \frac{1}{N} \sum_{x \in \Z} H \big(\tfrac{x}{N}\big)^2\Big)  \Big\{ 2^{(m+1)(\alpha-1)/2} + \sum_{j=0}^m 2^{j (\alpha-1)/2} \Big\}^2.\]
Since
\[\ell_0^{\alpha-1} \Big\{ 2^{(m+1)(\alpha-1)/2} + \sum_{j=0}^m 2^{j (\alpha-1)/2} \Big\}^2 \leq  C(\ell_0,\alpha) \beta_\alpha (\ell),\]
the proof is completed.
\end{proof}

Now, we are ready to prove Proposition \ref{prop:bg}.

\begin{proof}[Proof of Proposition \ref{prop:bg}]
By Cauchy-Schwarz inequality, if $\tilde{f} (\rhobf) = 0$, the expectation in \eqref{bg1} is bounded by
\begin{multline}\label{3.3}
	 3\E_{\nu_{\rhobf}} \Big[\sup_{0 \leq t \leq T} \Big( \int_0^t \sum_{x \in \Z} \Big[\tau_x f (\etabf_s)  - \E_{\nu_{\rhobf}} \big[\tau_x f (\etabf_s)| \etabf_s^{(\ell_0)}  (x)   \big]\Big] H_{\lambda,s} \big(\tfrac{x}{N}\big) ds \Big)^2 \Big] \\
	 +  3\E_{\nu_{\rhobf}} \Big[\sup_{0 \leq t \leq T} \Big( \int_0^t \sum_{x \in \Z} \Big[\E_{\nu_{\rhobf}} \big[\tau_x f (\etabf_s)| \etabf_s^{(\ell_0)}  (x)   \big]  - \E_{\nu_{\rhobf}} \big[\tau_x f (\etabf_s)| \etabf_s^{(\ell)}  (x)   \big]\Big] H_{\lambda,s} \big(\tfrac{x}{N}\big) ds \Big)^2 \Big] \\
	  +3\E_{\nu_{\rhobf}} \Big[\sup_{0 \leq t \leq T} \Big( \int_0^t \sum_{x \in \Z} \Big[\E_{\nu_{\rhobf}} \big[\tau_x f (\etabf_s)| \etabf_s^{(\ell)}  (x)   \big]  - \sum_{j=1}^n (\partial_{\rho^{j}} \tilde{f}) (\rhobf) \bar{\eta}^{j,(\ell)}_s (x) \Big] H_{\lambda,s} \big(\tfrac{x}{N}\big) ds \Big)^2 \Big].
\end{multline}
By Lemmas \ref{lem:3.3}, \ref{lem:3.5}, the first two lines in \eqref{3.3} are bounded by
\[  C(\rhobf,\ell_0,\alpha)  T \|f\|_{L^5 (\nu_{\rhobf})}^2 \frac{\ell_0^{1+\alpha} +\ell^\alpha}{N^{\alpha-1}} \Big(\frac{1}{N}\sum_{x \in \Z} H \big(\tfrac{x}{N}\big)^2\Big).\]
By Cauchy-Schwarz inequality, stationary and translation invariance of the measure $\nu_{\rhobf}$  and the equivalence of ensembles \eqref{equiofensem1}, the last line in \eqref{3.3} is bounded by
\begin{multline*}
	3 T^2 N^2 \big\| E_{\nu_{\rhobf}} \big[ f (\etabf)| \etabf^{(\ell)}  (0)   \big]  - \sum_{j=1}^n (\partial_{\rho^{j}} \tilde{f}) (\rhobf) \bar{\eta}^{j,(\ell)} (0)  \big] \big\|_{L^2 (\nu_{\rhobf})}^2 \Big(\frac{1}{N}\sum_{x \in \Z} \big| H_{\lambda,s} \big(\tfrac{x}{N}\big) \big|  \Big)^2 \\
	\leq C(\rhobf) T^2 \|f\|_{L^5 (\nu_{\rhobf})}^2 \frac{N^2}{\ell^2} \Big(\frac{1}{N}\sum_{x \in \Z} \big| H \big(\tfrac{x}{N}\big) \big|  \Big)^2.
\end{multline*}

The proof of \eqref{bg2} is similar. Instead of using \eqref{equiofensem1}, we use \eqref{equiofensem2} and bound
\begin{multline*}
	\E_{\nu_{\rhobf}} \Big[\sup_{0 \leq t \leq T} \Big( \int_0^t \sum_{x \in \Z} \Big[\E_{\nu_{\rhobf}} \big[\tau_x f (\etabf_s)| \etabf_s^{(\ell)}  (x)   \big]  \\
	- \frac{1}{2}\sum_{j,k=1}^n (\partial_{\rho^{j}} \partial_{\rho^{k}} \tilde{f}) (\rhobf) \big\{ \bar{\eta}^{j,(\ell)}_s (x) \bar{\eta}^{k,(\ell)}_s (x) - \tfrac{\Gamma_{j,k} (\rhobf)}{2\ell+1} \big\} \Big] H_{\lambda,s} \big(\tfrac{x}{N}\big) ds \Big)^2 \Big] \\
\leq 	C(\rhobf) T^2 \|f\|_{L^5 (\nu_{\rhobf})}^2 \frac{N^2}{\ell^3} \Big(\frac{1}{N}\sum_{x \in \Z} \big| H \big(\tfrac{x}{N}\big) \big|  \Big)^2.
\end{multline*}
This concludes the proof of the proposition.
\end{proof}

\section{Tightness}\label{sec:tightness}

In this section, we prove the tightness of the sequence $\{\mc{Y}^N_t, 0 \leq t \leq T\}_{N \geq 1}$. By Mitoma's criterion \cite{mitoma1983tightness},  we only need to prove that for any $1 \leq i \leq n$, for any $H \in \mathcal{S} (\R)$, the real valued sequence $\{\mathcal{Y}^{i,N}_t (H),\,0 \leq t \leq T\}_{N \geq 1}$ is tight.

For $H \in \mc{S} (\R)$, define
\begin{align*}
\Lcal_N H \big(\tfrac{x}{N}\big) &= N^\alpha \sum_{y \in \Z} p (y) \Big[ H \big(\tfrac{x+y}{N}\big) - H \big(\tfrac{x}{N}\big)\Big] - \frac{m^\alpha_N}{N} H^\prime \big(\tfrac{x}{N}\big),\\
\mc{U}_N H \big(\tfrac{x}{N}\big) &= N^\alpha \sum_{y \in \Z} p (y) \Big[ H \big(\tfrac{x+y}{N}\big) - H \big(\tfrac{x}{N}\big)\Big].
\end{align*}
Note that $\Lcal_N H = \mc{U}_N H$ if $0 < \alpha < 1$, and that $\Lcal_N$ could be viewed as a discrete version of the operator $\Lcal$ defined in \eqref{l}. In fact, by \cite[Proposition 2.2]{gonccalves2018density},
\begin{align}
	\lim_{N \rightarrow \infty} \sup_{x \in \Z} \Big|\Lcal_N H \big(\tfrac{x}{N}\big) - \Lcal H \big(\tfrac{x}{N}\big) \Big| = 0,\label{eqn0}\\
	\lim_{N \rightarrow \infty} \frac{1}{N}\sum_{x \in \Z} \Big|\Lcal_N H \big(\tfrac{x}{N}\big) - \Lcal H \big(\tfrac{x}{N}\big) \Big| = 0.\label{eqn1}
\end{align}

\subsection{Associated martingales}\label{subsec:martingale} We start from calculating the martingales associated to the fluctuation fields. By Dynkin's formula, for $H \in \mathcal{S} (\R)$ and for $1 \leq i \leq n$,
\begin{equation}\label{martingale1}
	\mc{M}_t^{i,N} (H) = \mathcal{Y}^{i,N}_t (H)  - \mathcal{Y}^{i,N}_0 (H)  - \int_0^t \big(N^\alpha \gen +\partial_s\big) \mathcal{Y}^{i,N}_s (H) ds 
\end{equation} 
is a martingale with quadratic variation
\begin{equation}\label{qv1}
	\<\mc{M}_\cdot^{i,N} (H)\>_t = \int_0^t \Big\{ N^\alpha \gen \mathcal{Y}^{i,N}_s (H)^2 - 2 \mathcal{Y}^{i,N}_s (H) N^\alpha \gen \mathcal{Y}^{i,N}_s (H)\Big\} ds.
\end{equation} 
Denote
\[V^i (\etabf) =  g_i (\etabf (0)) - \tilde{g}_i (\rhobf) 
- \sum_{j=1}^n (\partial_{\rho^{j}} \tilde{g}_i) (\rhobf) \bar{\eta}^j (0).\]
Direct calculations yield that the time integral in \eqref{martingale1} equals $\A_t^{i,N} (H) + \B_t^{i,N} (H)$, where
\begin{align*}
\A_t^{i,N} (H) &= \int_0^t \Big\{ \frac{N^\alpha}{\sqrt{N}} \sum_{x,y \in \Z} p(y-x) \tau_x V^i (\etabf_s) \Big[ H_{\lambda,s} \big(\tfrac{y}{N}\big) - H_{\lambda,s} \big(\tfrac{x}{N}\big)\Big] \Big\} \,ds\\
&= \int_0^t \Big\{ \frac{1}{\sqrt{N}} \sum_{x \in \Z}  \tau_x V^i  (\etabf_s) (\mc{U}_N H_{\lambda,s}) \big(\tfrac{x}{N}\big) \Big\} \,ds
\end{align*}
and
\begin{multline*}
\B_t^{i,N} (H)
	= \int_0^t \Big\{ \frac{N^\alpha}{\sqrt{N}} \sum_{x,y \in \Z} \sum_{j=1}^n p(y-x)   (\partial_{\rho^{j}} \tilde{g}_i) (\rhobf)\bar{\eta}^j_s (x) \Big[ H_{\lambda,s} \big(\tfrac{y}{N}\big) - H_{\lambda,s} \big(\tfrac{x}{N}\big)\Big] \\-  \frac{\lambda m_N^\alpha}{N^{3/2}} \sum_{x \in \Z} \bar{\eta}^i_s (x) H_{\lambda,s}^\prime \big(\tfrac{x}{N}\big) \Big\} \,ds,
\end{multline*}
and that
\begin{equation}\label{qv}
	\<\mc{M}_\cdot^{i,N} (H)\>_t = \int_0^t \Big\{ N^{\alpha-1} \sum_{x,y \in \Z} \big[ g_i (\etabf_s (x)) p (y-x) + g_i (\etabf_s (y)) p (x-y)\big] \big[ H_{\lambda,s} \big(\tfrac{y}{N}\big) - H_{\lambda,s} \big(\tfrac{x}{N}\big)\big]^2 \Big\} \,ds.
\end{equation}
Therefore,  to prove the tightness of the sequence $\{\mathcal{Y}^{i,N}_t (H),\,0 \leq t \leq T\}_{N \geq 1}$, we only need to prove the tightness of the following sequences,
\begin{align*}
	\{\mathcal{Y}^{i,N}_0 (H)\}_{N \geq 1}, \quad 	&\{\mc{M}_t^{i,N} (H), 0 \leq t \leq T\}_{N \geq 1},\\
	\{\B_t^{i,N} (H), 0 \leq t \leq T\}_{N \geq 1}, \quad &\{\A_t^{i,N} (H), 0 \leq t \leq T\}_{N \geq 1}.
\end{align*}

\subsection{Tightness of the  sequence $\{\mathcal{Y}^{i,N}_0 (H)\}_{N \geq 1}$.}  Since the invariant measure $\nu_{\rhobf}$ is product, by calculating the characteristic function of the initial density fluctuation field (see \cite[Lemma 11.2.1]{klscaling} for example), one could prove directly that $\mathcal{Y}^{N}_0 $ converges in distribution to $\mathcal{Y}_0 $, which  is Gaussian distributed with mean zero  and covariance
\[ \cov_{\nu_{\rhobf}} \Big(\mathcal{Y}^i_0(H),\mathcal{Y}^j_0 (G)\Big) = \Gamma_{i,j} (\rhobf) \<H,G\>, \quad H,G \in \mathcal{S} (\R).\]
In particular, the sequence $\{\mathcal{Y}^{i,N}_0 (H)\}_{N \geq 1}$ is tight.

\subsection{Tightness of the sequence $\{\mc{M}_t^{i,N} (H), 0 \leq t \leq T\}_{N \geq 1}$.}  This follows immediately from the following result.

\begin{lemma}\label{lem:ConvMart}
The sequence of martingales $\{\mc{M}_t^{i,N} (H), 0 \leq t \leq T\}_{N \geq 1}$ converges in distribution, as $N \rightarrow \infty$, to $\{\mc{M}_t^{i} (H), 0 \leq t \leq T\}$, which is a Brownian motion with variance
\begin{equation}\label{var}
4 \tilde{g}_i (\rhobf) \<H,(-\mc{S})H\> = (c_++c_-) \tilde{g}_i (\rhobf)  \iint_{\R^2} \frac{\big( H (u) - H (v) \big)^2}{|u-v|^{1+\alpha}}  \,du \,dv.
\end{equation}
Moreover, for $H,G \in \mc{S} (\R)$, the cross variation between  $\mc{M}_t^{i} (H)$ and $\mc{M}_t^{j} (G)$ vanishes if $i \neq j$.
\end{lemma}

\begin{proof}
For the first statement, by \cite[Theorem 2.1]{whitt2007proofs}, we only need to prove
	\begin{enumerate}
		\item for any $\varepsilon > 0$, 
		\begin{equation}\label{tightM1}
			\lim_{N \rightarrow \infty} \bb{P}_{\nu_{\rhobf}} \Big( \sup_{0 \leq t \leq T} \big| \mc{M}_t^{i,N} (H) - \mc{M}_{t-}^{i,N} (H) \big| > \varepsilon \Big) = 0;
		\end{equation}
	\item for any $0 \leq t \leq T$,
	\begin{equation}\label{tightM2}
		\lim_{N \rightarrow \infty} \<\mc{M}_\cdot^{i,N} (H)\>_t = 4 t \tilde{g}_i (\rhobf) \<H,(-\mc{S})H\>
	\end{equation}
in distribution. 
	\end{enumerate}
Since  there is a jump in the martingale if and only if there is one in the density fluctuation field, and at most one particle could jump at any time,
\[\sup_{0 \leq t \leq T} \big| \mc{M}_t^{i,N} (H) - \mc{M}_{t-}^{i,N} (H) \big| = \sup_{0 \leq t \leq T} \big| \mc{Y}_t^{i,N} (H) - \mc{Y}_{t-}^{i,N} (H) \big| \leq \frac{2 \|H\|_{L^\infty (\R)}}{\sqrt{N}},\]
which proves \eqref{tightM1}.  By \eqref{qv}, 
\[\lim_{N \rightarrow \infty} \E_{\nu_{\rhobf}} \big[ \<\mc{M}_\cdot^{i,N} (H)\>_t \big]  = 4 t \tilde{g}_i (\rhobf) \<H,(-\mc{S})H\>\]
and by Cauchy-Schwarz inequality,
\begin{multline*}
	{\rm Var}_{\nu_{\rhobf}} \Big(\<\mc{M}_\cdot^{i,N} (H)\>_t\Big) \leq C(\rhobf) t^2 N^{2\alpha-2}  \Big\{  \sum_{x \in \Z}   \Big( \sum_{y \in \Z} p(y-x) \big[ H \big(\tfrac{y}{N}\big) - H \big(\tfrac{x}{N}\big)\big]^2 \Big)^2 \\
	+ \sum_{y \in \Z} \Big( \sum_{x \in \Z} p(x-y) \big[ H \big(\tfrac{y}{N}\big) - H \big(\tfrac{x}{N}\big)\big]^2 \Big)^2 \Big\}\\
	= 2 C(\rhobf) t^2 N^{2\alpha-2}  \sum_{x \in \Z}   \Big( \sum_{y \in \Z} p(y-x) \big[ H \big(\tfrac{y}{N}\big) - H \big(\tfrac{x}{N}\big)\big]^2 \Big)^2 .
\end{multline*}
Re-index $y$ by $x+y$, we bound the last line by 
\begin{multline}\label{eqn7}
C(\rhobf) t^2 N^{-2}  \int_{\tfrac{1}{N}}^\infty \Big( \int_{\tfrac{1}{N}}^\infty  \frac{\big(H(u+v) -H(u)\big)^2}{|v|^{1+\alpha}} \,dv \Big)^2\,du\\
\leq C(\rhobf) t^2 N^{-2}  \int_{\tfrac{1}{N}}^\infty \Big( \int_{\tfrac{1}{N}}^1  \frac{\big(H(u+v) -H(u)\big)^2}{|v|^{1+\alpha}} \,dv \Big)^2\,du\\
+  C(\rhobf) t^2 N^{-2}  \int_{\tfrac{1}{N}}^\infty \Big( \int_{1}^\infty  \frac{\big(H(u+v) -H(u)\big)^2}{|v|^{1+\alpha}} \,dv \Big)^2\,du.
\end{multline}
For $H \in \mc{S}(\R)$, denote
\[\|H (u)\|_{1,\infty} = \sup_{|v-u| \leq 1} |H(v)|.\]
We bound the penultimate line in \eqref{eqn7} by
\[C(\rhobf) t^2 N^{-2} \int_{\tfrac{1}{N}}^\infty  \|H^\prime (u)\|_{1,\infty}^2  \,du \Big( \int_{\tfrac{1}{N}}^1  \frac{1}{|v|^{\alpha-1}} \,dv \Big)^2 \leq C(\rhobf,H,\alpha) t^2 N^{-2}\]
since $\alpha \leq 3/2$. By using  Cauchy-Schwarz inequality twice, we bound the last line in \eqref{eqn7} by
\[C(\rhobf) t^2 N^{-2}  \Big(\int_{1}^\infty \frac{1}{|v|^{1+\alpha}} \,dv \Big) \int_{\tfrac{1}{N}}^\infty \int_{1}^\infty \frac{H(u+v)^4 + H(u)^4}{|v|^{1+\alpha}} \,dv\,du  \leq C(\rhobf,H,\alpha) t^2 N^{-2} \] 
since $\alpha > 0$. This completes the proof of \eqref{tightM2}.

For the second statement, since an $i$--type and a $j$--type particle cannot jump simultaneously,  direct calculations show that
\begin{multline*}
	\<\mc{M}^{i,N}_\cdot (H), \mc{M}^{j,N}_\cdot (G)\>_t = \int_0^t \Big\{ N^\alpha \gen \big(\mathcal{Y}^{i,N}_s (H)  \mathcal{Y}^{j,N}_s (G)\big) -  \mathcal{Y}^{i,N}_s (H) N^\alpha \gen \mathcal{Y}^{j,N}_s (G) \\- \mathcal{Y}^{j,N}_s (G) N^\alpha \gen \mathcal{Y}^{i,N}_s (H)  \Big\} ds = 0.
\end{multline*}
We conclude the proof by letting $N \rightarrow \infty$.
\end{proof}

\subsection{Tightness of the sequence $\{\B_t^{i,N} (H), 0 \leq t \leq T\}_{N \geq 1}$.}\label{subsec:tightB} By Kolmogorov-Centsov's tightness criterion (see \cite[Proposition 4.3]{gonccalves2018density} for example), we only need  to show
\[\E_{\nu_{\rhobf}} \Big[ \big| \B_t^{i,N} (H) - \B_s^{i,N} (H)\big|^a \Big] \leq C |t-s|^{1+b}\]
for some constants $C,a,b > 0$. 

If $0 < \alpha < 1$, recall $m_N^\alpha = 0$. Thus, we could rewrite 
\[\B_t^{i,N} (H) = \int_0^t \sum_{j=1}^n \big(\partial_{\rho^{j}} \tilde{g}_i \big) (\rhobf) \mc{Y}_s^{j,N} \big(\Lcal_N H\big) ds.\]
If $1 \leq \alpha \leq 3/2$, since the density  $\rhobf$ satisfies Assumption \ref{assump:frame}, we have 
\[\B_t^{i,N} (H) = \lambda \int_0^t  \mc{Y}_s^{i,N} \big(\Lcal_N H \big) ds.\]
In both cases, by Cauchy-Schwarz inequality and stationary of the process,
\[\E_{\nu_{\rhobf}} \Big[ \big| \B_t^{i,N} (H) - \B_s^{i,N} (H)\big|^2 \Big] \leq C(\rhobf,n) |t-s|^{2} \frac{1}{N} \sum_{x \in \Z}  \big(\Lcal_N H (\tfrac{x}{N})\big)^2 \leq C(\rhobf,n) \|\Lcal H\|^2_{L^2 (\R)} |t-s|^{2}.\]
In the last inequality, we use \eqref{eqn0} and \eqref{eqn1}. Thus, the sequence $\{\B_t^{i,N} (H), 0 \leq t \leq T\}_{N \geq 1}$ is tight.

\subsection{Tightness of the sequence $\{\A_t^{i,N} (H), 0 \leq t \leq T\}_{N \geq 1}$.}\label{subsec:tightA} We first state a technical lemma before proving the tightness of the sequence $\{\A_t^{i,N} (H), 0 \leq t \leq T\}_{N \geq 1}$. 

\begin{lemma}\label{lem1}
Recall the definition of the operator $\mc{U}_N$ in the beginning of this section. For any $H \in \mathcal{S} (\R)$, there exists some constant $C=C(H,c_+,c_-,\alpha)$ such that
\[\frac{1}{N} \sum_{x \in \Z} \big| \mc{U}_N H \big(\tfrac{x}{N} \big) \big|  \leq C \gamma_\alpha (N),  \qquad \frac{1}{N} \sum_{x \in \Z} \Big[ \mc{U}_N H \big(\tfrac{x}{N}\big)\Big]^2 \leq C \gamma_\alpha (N)^2,\]
where
\[\gamma_\alpha (N) = \begin{cases}
	1, &\quad 0<  \alpha < 1;\\
	\log N, &\quad \alpha=1;\\
	N^{\alpha-1}, &\quad \alpha > 1. 
\end{cases}\]
\end{lemma}

\begin{proof}
For the first inequality in the lemma, we bound the left side by a constant multiple of 
\begin{equation}\label{eqn2}
	\iint_{\{|v|>1,u \in \R\}} \frac{|H(u+v)-H(u)|}{|v|^{1+\alpha}} \,du\,dv + \iint_{\{N^{-1} \leq |v| \leq 1,u \in \R\}} \frac{|H(u+v)-H(u)|}{|v|^{1+\alpha}} \,du\,dv.
\end{equation}
Since $H \in \mathcal{S} (\R)$ and $\alpha > 0$, the first term in \eqref{eqn2} is bounded by some constant.  We bound the second term in \eqref{eqn2} by
\[2 \int_\R \|H^\prime (u)\|_{1,\infty} du \int_{\tfrac{1}{N}}^1 v^{-\alpha}\,dv \leq C \gamma_\alpha (N). \]

The second inequality in the lemma could be proved in the same way. By Cauchy-Schwarz inequality, we bound the left side by a constant multiple of 
\begin{multline*}
	\int_{\R} \Big\{ \int_{1}^\infty \frac{|H(u+v)-H(u)|}{|v|^{1+\alpha}} \,dv \Big\}^2\,du + 	\int_{\R} \Big\{ \int_{\tfrac{1}{N}}^1 \frac{|H(u+v)-H(u)|}{|v|^{1+\alpha}} \,dv \Big\}^2\,du \\
	\leq C + C \int_\R \|H^\prime (u)\|_{1,\infty}^2 du \Big\{\int_{\tfrac{1}{N}}^1 |v|^{-\alpha}\,dv\Big\}^2 \leq C \gamma_\alpha (N)^2.
\end{multline*}
This proves the lemma.
\end{proof}

\subsubsection{The case $0 < \alpha \leq 1$.} Recall the expression of $\A_t^{i,N} (H)$ in Subsection \ref{subsec:martingale}. On one hand, since $\tilde{V}^i (\rhobf) = 0$, by using \eqref{bg1} and Lemma \ref{lem1}, we bound
\begin{multline}\label{eqn3}
\E_{\nu_{\rhobf}} \Big[ \big| \A_t^{i,N} (H) - \A_s^{i,N} (H)\big|^2 \Big] \\
\leq  \frac{C(\rhobf,\alpha)}{N}  \Big\{ \frac{|t-s|\ell^\alpha}{N^{\alpha-1}} \Big( \frac{1}{N} \sum_{x \in \Z} \big[ \mc{U}_N H \big(\tfrac{x}{N}\big)\big]^2\Big)  + \frac{|t-s|^2 N^2}{\ell^2} \Big(\frac{1}{N} \sum_{x \in \Z} \big| \mc{U}_N H \big(\tfrac{x}{N} \big) \big|\Big)^2\Big\}\\
\leq C(\rhobf,\alpha,H) \gamma_\alpha (N)^2 \Big\{ \frac{|t-s| \ell^\alpha}{N^\alpha} + \frac{|t-s|^2N}{\ell^2}\Big\}.
\end{multline}
On the other hand, by Cauchy-Schwarz inequality and Lemma \ref{lem1}, we bound
\begin{equation}\label{eqn4}
\E_{\nu_{\rhobf}} \Big[ \big| \A_t^{i,N} (H) - \A_s^{i,N} (H)\big|^2 \Big] \leq \frac{C(\rhobf) |t-s|^2}{N} \sum_{x \in \Z} \big[ \mc{U}_N H \big(\tfrac{x}{N}\big)\big]^2 \leq C(\rhobf,\alpha,H) \gamma_\alpha (N)^2 |t-s|^2.
\end{equation}
If $N^{(1+\alpha)/(2+\alpha)} |t-s|^{1/(2+\alpha)} \geq 1$, we use \eqref{eqn3} and take $\ell = N^{(1+\alpha)/(2+\alpha)} |t-s|^{1/(2+\alpha)}$, then
\[\E_{\nu_{\rhobf}} \Big[ \big| \A_t^{i,N} (H) - \A_s^{i,N} (H)\big|^2 \Big] \leq  C(\rhobf,\alpha,H)  \gamma_\alpha (N)^2 N^{-\tfrac{\alpha}{2+\alpha}} |t-s|^{1+\tfrac{\alpha}{2+\alpha}}.\]
Otherwise,  we use \eqref{eqn4} and bound
\[\E_{\nu_{\rhobf}} \Big[ \big| \A_t^{i,N} (H) - \A_s^{i,N} (H)\big|^2 \Big] \leq  C(\rhobf,\alpha,H)  \gamma_\alpha (N)^2 N^{-\tfrac{2 (1 +\alpha)}{2+\alpha}} |t-s|^{1+\tfrac{\alpha}{2+\alpha}}.\]
In both cases, 
\[\sup_{N \geq 1} \E_{\nu_{\rhobf}} \Big[ \big| \A_t^{i,N} (H) - \A_s^{i,N} (H)\big|^2 \Big] \leq  C(\rhobf,\alpha,H) |t-s|^{1+\tfrac{\alpha}{2+\alpha}}.\]
Therefore, by Kolmogorov-Centsov's tightness criterion, the sequence $\{\A_t^{i,N} (H), 0 \leq t \leq T\}_{N \geq 1}$ is tight if $0< \alpha \leq 1$ .

\subsubsection{The case $\alpha > 1$} Since $\nabla \tilde{V}^i (\rhobf) = 0$,  by \eqref{bg2}, Cauchy-Schwarz inequality and stationary of the process, we bound 
\begin{multline}\label{eqn6}
	\E_{\nu_{\rhobf}} \Big[ \big| \A_t^{i,N} (H) - \A_s^{i,N} (H)\big|^2 \Big] \leq \frac{C (\rhobf,\alpha)}{N} \Big\{ \\ |t-s|^2  \E_{\nu_{\rhobf}} \Big[
	\Big(\frac{1}{2} \sum_{x \in \Z} \Big[
	 \sum_{j,k=1}^n (\partial_{\rho^{j}} \partial_{\rho^{k}}  \tilde{V}^i) (\rhobf) \big\{\bar{\eta}^{j,(\ell)} (x) \bar{\eta}^{k,(\ell)} (x) - \tfrac{\Gamma_{j,k} (\rhobf)}{2\ell+1}  \big\}\Big] \mc{U}_N H_{\lambda,s} \big(\tfrac{x}{N}\big)  \Big)^2 \Big] \\ + \frac{|t-s|\beta_\alpha (\ell)}{N^{\alpha-1}} \Big(\frac{1}{N} \sum_{x \in \Z} \mc{U}_N H  \big(\tfrac{x}{N}\big)^2\Big) + \frac{|t-s|^2 N^2}{ \ell^3} \Big(\frac{1}{N} \sum_{x \in \Z} \mc{U}_N H  \big(\tfrac{x}{N}\big) \Big)^2 \Big\}.
\end{multline}
By Cauchy-Schwarz inequality and translation invariance, we bound
\begin{multline*}
	\E_{\nu_{\rhobf}} \Big[
	\Big(\frac{1}{2} \sum_{x \in \Z} \Big[
	\sum_{j,k=1}^n (\partial_{\rho^{j}} \partial_{\rho^{k}}  \tilde{V}^i) (\rhobf) \big\{\bar{\eta}^{j,(\ell)} (x) \bar{\eta}^{k,(\ell)} (x) - \tfrac{\Gamma_{j,k} (\rhobf)}{2\ell+1}  \big\}\Big] \mc{U}_N H_{\lambda,s} \big(\tfrac{x}{N}\big)  \Big)^2 \Big] \\
	\leq  C (\rhobf)
   \Big\{  \sum_{x \in \Z} \sum_{j,k=1}^n 
	\E_{\nu_{\rhobf}} \big[ \big\{\bar{\eta}^{j,(\ell)} (0) \bar{\eta}^{k,(\ell)} (0) - (2\ell+1)^{-1} \Gamma_{j,k} (\rhobf) \big\}^2\big]^{1/2} \big| \mc{U}_N H  \big(\tfrac{x}{N}\big)\big| \Big\}^2 \\
	 \leq   \frac{C (\rhobf,n) N^2}{\ell^2}   \Big\{  \frac{1}{N}\sum_{x \in \Z}  \big| \mc{U}_N H  \big(\tfrac{x}{N}\big) \big| \Big\}^2.
\end{multline*}
By Lemma \ref{lem1}, since $\ell \geq 1$, we bound
\begin{multline}\label{eqn5}
	\E_{\nu_{\rhobf}} \Big[ \big| \A_t^{i,N} (H) - \A_s^{i,N} (H)\big|^2 \Big] \leq 
 C (\rhobf,\alpha,H,n) \gamma_\alpha (N)^2 \Big\{ \frac{|t-s|^2 N}{\ell^2} + \frac{|t-s|\beta_\alpha (\ell)}{N^\alpha}  \Big\} \\
 = C (\rhobf,\alpha,H,n) N^{2\alpha-2} \Big\{ \frac{|t-s|^2 N}{\ell^2} + \frac{|t-s| \ell^{\alpha-1} }{N^\alpha}  \Big\}.
\end{multline}
As in the case $0 < \alpha \leq 1$, if $N |t-s|^{1/(1+\alpha)} \geq 1$, take $\ell = N |t-s|^{1/(1+\alpha)} \geq 1$ and use \eqref{eqn5},  then
\[\E_{\nu_{\rhobf}} \Big[ \big| \A_t^{i,N} (H) - \A_s^{i,N} (H)\big|^2 \Big] \leq  C (\rhobf,\alpha,H,n)  N^{2\alpha-3} |t-s|^{\tfrac{2\alpha}{1+\alpha}}.\]
Otherwise, we use \eqref{eqn4} and bound
\[\E_{\nu_{\rhobf}} \Big[ \big| \A_t^{i,N} (H) - \A_s^{i,N} (H)\big|^2 \Big] \leq  C (\rhobf,\alpha,H,n)  N^{2\alpha-4} |t-s|^{\tfrac{2\alpha}{1+\alpha}}.\]
Since $\alpha \leq 3/2$,
\[\sup_{N \geq 1} \E_{\nu_{\rhobf}} \Big[ \big| \A_t^{i,N} (H) - \A_s^{i,N} (H)\big|^2 \Big] \leq C (\rhobf,\alpha,H,n)   |t-s|^{\tfrac{2\alpha}{1+\alpha}}.\]
Note that $2\alpha/(1+\alpha) > 1$ if $\alpha > 1$. Therefore,  the sequence $\{\A_t^{i,N} (H), 0 \leq t \leq T\}_{N \geq 1}$ is tight if $1 < \alpha \leq 3/2$.

\section{Characterization of the limit points}\label{sec:clp}

In the last section, we have shown the tightness of the sequences $\{\mc{Y}^N_t\}$, $\{\mc{M}^N_t\}$, $\{\A^N_t\}$ and $\{\B^N_t\}$.  Denote $\{\mc{Y}_t\}$, $\{\mc{M}_t\}$, $\{\mc{A}_t\}$ and $\{\mc{B}_t\}$ any limit point of the above terms along some subsequence.  Since the zero range process is stationary, for any $0 \leq t \leq T$, $\mc{Y}_t$ is also stationary. In the rest of this section, we characterize the martingale $\{\mc{M}_t\}$.

\subsection{The case $0 < \alpha < 1$.} Take $H \in \mc{S} (\R)$. By the computations in Section \ref{sec:tightness}, 
\begin{equation}\label{clp1}
\mc{M}_t^{i,N} (H) = \mc{Y}_t^{i,N} (H) - \mc{Y}_0^{i,N} (H) - \A_t^{i,N} (H) - \int_0^t \sum_{j=1}^n (\partial_{\rho^{j}} \tilde{g}_i) (\rhobf) \mc{Y}_s^{j,N} (\Lcal_N H) \,ds.
\end{equation}
We have shown in \eqref{eqn3} that
\[\E_{\nu_{\rhobf}} \Big[\sup_{0 \leq t \leq T} \big(\A_t^{i,N} (H)\big)^2 \Big] \leq C(\rhobf,\alpha,H) \Big\{  \frac{T\ell^\alpha}{N^\alpha} + \frac{T^2 N}{\ell^2}\Big\}.\]
Taking $\ell = N^{(1+\alpha)/(2+\alpha)}$, we bound
\[\E_{\nu_{\rhobf}} \Big[\sup_{0 \leq t \leq T} \big(\A_t^{i,N} (H)\big)^2 \Big] \leq C(\rhobf,\alpha,H,T) N^{-\tfrac{\alpha}{\alpha+2}}.\]
Therefore, $\A_t^{i,N} (H)$ vanishes in the limit.  For the fourth term on the right side of \eqref{clp1},  note that  for any $1 \leq j \leq n$,
\[\lim_{N \rightarrow \infty} \E_{\nu_{\rhobf}} \Big[ \big\{\mc{Y}_s^{j,N} (\Lcal_N H) - \mc{Y}_s^{j,N} (\Lcal H)\}^2 \Big] \leq \lim_{N \rightarrow \infty} \frac{C(\rhobf)}{N} \sum_{x \in \Z} \big| \mc{L}_N H (\tfrac{x}{N}) - \mc{L} H (\tfrac{x}{N})\big|^2 = 0.\]
In the last identity we use \eqref{eqn0} and \eqref{eqn1}.  Therefore, by Cauchy-Schwarz inequality,  one could replace $\mc{L}_N H$ with $\mc{L} H$ in the fourth term on the right side of \eqref{clp1}.  Passing to the limit, we have
\[\mc{M}_t^{i} (H) = \mc{Y}_t^{i} (H) - \mc{Y}_0^{i} (H)  - \int_0^t \sum_{j=1}^n (\partial_{\rho^{j}} \tilde{g}_i) (\rhobf) \mc{Y}_s^{j} (\mc{L} H) \,ds.\]
By Lemma \ref{lem:ConvMart} and Levy's theorem , $\mc{M}_t^{i} (H)$ is a Brownian motion with variance given by \eqref{var}. Moreover,  for $i \neq j$ and $H,G \in \mc{S} (\R)$, $\mc{M}_t^{i} (H)$ and $\mc{M}_t^{j} (G)$ are independent. This concludes the proof if $0 < \alpha < 1$.

\begin{remark}
	Note that in the Definitions \ref{def:ou} and \ref{def:fburgers}, the test functions are taken to be $\mc{S}(\R)$--valued trajectories $H: [0,T] \rightarrow \R$. However, by the techniques presented in \cite[Subsection 5.1]{gonccalves2018density}, one could extend the results from functions in $\mc{S} (\R)$ to $\mc{S}(\R)$--valued trajectories. Therefore, to make ideas clear, we assume $H \in \mc{S} (\R)$ above and hereafter.
\end{remark}

\subsection{The case $1 \leq  \alpha < 3/2$.} Since $\rhobf$ satisfies the Assumption \ref{assump:frame},
\begin{equation}\label{clp2}
	\mc{M}_t^{i,N} (H) = \mc{Y}_t^{i,N} (H) - \mc{Y}_0^{i,N} (H) - \A_t^{i,N} (H) - \lambda \int_0^t  \mc{Y}_s^{i,N} (\Lcal_N H) \,ds.
\end{equation}
In \eqref{eqn5}, we have shown that\footnote{Note that \eqref{eqn5} also holds for $\alpha = 1$.}
\[\E_{\nu_{\rhobf}} \Big[\sup_{0 \leq t \leq T} \big(\A_t^{i,N} (H)\big)^2 \Big] \leq C (\rhobf,\alpha,H,n) \gamma_\alpha (N)^2 \Big\{ \frac{T^2N}{\ell^2} + \frac{T \beta_\alpha (\ell)}{N^\alpha}\Big\}.\]
Take $\ell = N$, then
\[\E_{\nu_{\rhobf}} \Big[\sup_{0 \leq t \leq T} \big(\A_t^{i,N} (H)\big)^2 \Big] \leq C (\rhobf,\alpha,H,n,T) \big[ N^{2 \alpha -3} + (\log N)^4/N\big].\]
Therefore, $\A_t^{i,N} (H)$  converges to zero as $N \rightarrow \infty$. Replacing  $\Lcal_N H$ with $\Lcal H$ and passing to the limit, we have
\[\mc{M}_t^{i} (H) = \mc{Y}_t^{i} (H) - \mc{Y}_0^{i} (H)  - \lambda \int_0^t  \mc{Y}_s^{i} (\Lcal H) \,ds.\]
We conclude the proof in the case $1 \leq  \alpha < 3/2$ by  Lemma \ref{lem:ConvMart}.

\subsection{The case $\alpha = 3/2$.} In this case, \eqref{clp2} still holds true, but the term  $\A_t^{i,N} (H)$ does not vanish in the limit.  By \eqref{bg2} and \eqref{eqn6},
	\begin{multline}\label{clp3}
	\E_{\nu_{\rhobf}} \Big[\sup_{0 \leq t \leq T} \Big(\A_t^{i,N} (H) \\
	- \frac{1}{2\sqrt{N}} \int_0^t \sum_{x \in \Z}
	 \sum_{j,k=1}^n (\partial_{\rho^{j}} \partial_{\rho^{k}}  \tilde{V}^i) (\rhobf) \big\{\bar{\eta}^{j,(\ell)}_s (x)  \bar{\eta}^{k,(\ell)}_s (x) - 
	 \frac{1}{2\ell+1} \Gamma_{j,k} (\rhobf) \big\} \mc{U}_N H_{\lambda,s} \big(\tfrac{x}{N}\big) ds \Big)^2 \Big] \\
	\leq C(\rhobf,\alpha,H) \Big\{ \frac{T \ell^{1/2}}{N^{1/2}}  + \frac{T^2N^2}{\ell^3} \Big\}.
\end{multline}
From the definitions of $\Lcal_N$ and $\mc{U}_N$, if $\alpha=2/3$,  for $H \in \mc{S} (\R)$, \[\Lcal_N H = \mc{U}_N H -  m N^{1/2} H^\prime.\] 
Note that $(\partial_{\rho^{j}} \partial_{\rho^{k}}  \tilde{V}^i) (\rhobf) = (\partial_{\rho^{j}} \partial_{\rho^{k}}  \tilde{g}^i) (\rhobf)$. By Cauchy-Schwarz inequality,
	\begin{multline}\label{clp6}
	\E_{\nu_{\rhobf}} \Big[\sup_{0 \leq t \leq T} \Big(\A_t^{i,N} (H) 
	- \frac{m}{2} \int_0^t \sum_{x \in \Z}
	\sum_{j,k=1}^n (\partial_{\rho^{j}} \partial_{\rho^{k}}  \tilde{g}_i) (\rhobf) \bar{\eta}^{j,(\ell)}_s (x)  \bar{\eta}^{k,(\ell)}_s (x)  H_{\lambda,s}^\prime \big(\tfrac{x}{N}\big) ds \Big)^2 \Big] \\
	\leq 2 \E_{\nu_{\rhobf}} \Big[\sup_{0 \leq t \leq T} \Big( \frac{1}{2\sqrt{N}} \int_0^t \sum_{x \in \Z}
	\sum_{j,k=1}^n (\partial_{\rho^{j}} \partial_{\rho^{k}}  \tilde{g}_i) (\rhobf) \big\{\bar{\eta}^{j,(\ell)}_s (x)  \bar{\eta}^{k,(\ell)}_s (x) 
	- 
	\frac{\Gamma_{j,k} (\rhobf)}{2\ell+1}  \big\} \Lcal_N H_{\lambda,s} \big(\tfrac{x}{N}\big) ds \Big)^2 \Big]  \\
	+ 2 \E_{\nu_{\rhobf}} \Big[\sup_{0 \leq t \leq T} \Big( \frac{m}{2} \int_0^t \sum_{x \in \Z}
	\sum_{j,k=1}^n  (\partial_{\rho^{j}} \partial_{\rho^{k}}  \tilde{g}_i) (\rhobf) \frac{1}{2\ell+1} \Gamma_{j,k} (\rhobf) H_{\lambda,s}^\prime \big(\tfrac{x}{N}\big) ds \Big)^2 \Big].
\end{multline}
By Cauchy-Schwarz inequality, stationary of the process and \eqref{eqn1}, we bound the first term on the right side of \eqref{clp6} by
\begin{multline*}
 \frac{C (\rhobf)T^2}{N} \Big\{ \sum_{x \in \Z}
\sum_{j,k=1}^n  \E_{\nu_{\rhobf}} \Big[ \Big(\bar{\eta}^{j,(\ell)} (0)  \bar{\eta}^{k,(\ell)} (0) 
- 
\frac{1}{2\ell+1} \Gamma_{j,k} (\rhobf)\Big)^2  \Big]^{1/2} \big|\Lcal_N H \big(\tfrac{x}{N}\big)\big|\Big\}^2 \\
\leq \frac{C (\rhobf,n) T^2 N}{\ell^2} \|\Lcal H\|_{L^1 (\R)}^2.
\end{multline*}
Since 
\[H^\prime \big(\tfrac{x}{N}\big)= N  \big[H\big(\tfrac{x+1}{N}\big)-H\big(\tfrac{x}{N}\big)\big] - \frac{1}{2N} H^{\prime\prime} \big(\tfrac{\xi}{N}\big)\]
for some $\xi \in (x,x+1)$, and 
\[\sum_{x \in \Z} N  \big[H\big(\tfrac{x+1}{N}\big)-H\big(\tfrac{x}{N}\big)\big] = 0,\]
 the second term on the right side of \eqref{clp6} is bounded by
\[\frac{C (\rhobf,n,m)T^2}{\ell^2} \|H^{\prime\prime}\|_{L^1(\R)}^2.\]
Adding up the above estimates,
	\begin{multline*}
	\E_{\nu_{\rhobf}} \Big[\sup_{0 \leq t \leq T} \Big(\A_t^{i,N} (H) 
	- \frac{m}{2} \int_0^t \sum_{x \in \Z}
	\sum_{j,k=1}^n (\partial_{\rho^{j}} \partial_{\rho^{k}}  \tilde{g}_i) (\rhobf) \bar{\eta}^{j,(\ell)}_s (x)  \bar{\eta}^{k,(\ell)}_s (x)  H_{\lambda,s}^\prime \big(\tfrac{x}{N}\big) ds \Big)^2 \Big] \\
	\leq C\Big\{ \frac{T \ell^{1/2}}{N^{1/2}}  + \frac{T^2N^2}{\ell^3} + \frac{T^2 N}{\ell^2}\Big\}.
\end{multline*}
for some constant $C = C (\rhobf,H,n,m)$. Take $\ell = \varepsilon N$ for $\varepsilon > 0$ in the above inequality,
	\begin{multline}\label{clp4}
	\limsup_{N \rightarrow \infty} \E_{\nu_{\rhobf}} \Big[\sup_{0 \leq t \leq T} \Big(\A_t^{i,N} (H) 
	- \frac{m}{2} \int_0^t \sum_{x \in \Z}
	\sum_{j,k=1}^n (\partial_{\rho^{j}} \partial_{\rho^{k}}  \tilde{g}_i) (\rhobf) \bar{\eta}^{j,(\varepsilon N)}_s (x)  \bar{\eta}^{k,(\varepsilon N)}_s (x)  H_{\lambda,s}^\prime \big(\tfrac{x}{N}\big) ds \Big)^2 \Big] \\
	\leq C T \varepsilon^{1/2}.
\end{multline}
For $1 \leq j,k \leq n$, define
\[\mc{A}^{j,k,\varepsilon,N}_t (H) = \int_0^t \frac{1}{N} \sum_{x \in \Z} \tau_{-x} \mc{Y}_s^{j,N} (G_\varepsilon) \tau_{-x}  \mc{Y}_s^{k,N} (G_\varepsilon) H^\prime \big(\tfrac{x}{N}\big) ds.\]
Recall $G_\varepsilon$ is an approximation of the function $\iota(\cdot)$ defined in \eqref{iota}.  Since $\{\mc{Y}_t^N (G_\varepsilon), 0 \leq t \leq T\}_{N \geq 1}$ is tight, the sequence  $\{\mc{A}^{j,k,\varepsilon,N}_t (H), 0 \leq t \leq T\}_{N \geq 1}$ is also tight, and denote $\{\mc{A}^{j,k,\varepsilon}_t (H)\}$ any limit point along some subsequence.  Passing to the limit,
\[\mc{A}^{j,k,\varepsilon}_t (H) = \int_0^t \int_{\R} \tau_{-u} \mc{Y}_s^{j} (G_\varepsilon) \tau_{-u} \mc{Y}_s^{k} (G_\varepsilon) H^\prime (u)\,du\,ds.\]
Note that 
\[\sum_{x \in \Z} \bar{\eta}^{j,(\varepsilon N)}_s (x)  \bar{\eta}^{k,(\varepsilon N)}_s (x)  H_{\lambda,s}^\prime \big(\tfrac{x}{N}\big) = \frac{1}{N}  \sum_{x \in \Z} \tau_{-x} \mc{Y}_s^{j,N} (\iota_\varepsilon) \tau_{-x}  \mc{Y}_s^{k,N} (\iota_\varepsilon) H^\prime \big(\tfrac{x}{N}\big) \]
plus an error term of order $\mathcal{O} (N^{-1})$. Moreover, by Cauchy-Schwarz inequality,
	\begin{multline*}
	\E_{\nu_{\rhobf}} \Big[\sup_{0 \leq t \leq T} \Big(\mc{A}^{j,k,\varepsilon,N}_t (H) - \int_0^t \frac{1}{N}  \sum_{x \in \Z} \tau_{-x} \mc{Y}_s^{j,N} (\iota_\varepsilon) \tau_{-x} \mc{Y}_s^{k,N} (\iota_\varepsilon) H^\prime \big(\tfrac{x}{N}\big) ds 
	 \Big)^2 \Big] \\
	\leq C(\rhobf) T^2 \|H^\prime\|_{L^1 (\R)}^2 \varepsilon^{-1} \|G_\varepsilon - \iota_\varepsilon\|_{L^2 (\R)}^2.
\end{multline*}
Therefore, by \eqref{clp4}, 
	\begin{multline}\label{clp5}
	\limsup_{N \rightarrow \infty} \E_{\nu_{\rhobf}} \Big[\sup_{0 \leq t \leq T} \Big(\A_t^{i,N} (H) 
	- \frac{m}{2} \sum_{j,k=1}^n (\partial_{\rho^{j}} \partial_{\rho^{k}}  \tilde{g}_i) (\rhobf) \mc{A}^{j,k,\varepsilon,N}_t (H) \Big)^2 \Big]  \\
	\leq C\Big\{ T \varepsilon^{1/2} + T^2 \varepsilon^{-1} \|G_\varepsilon - \iota_\varepsilon\|_{L^2 (\R)}^2 \Big\}
\end{multline}
for some constant $C=C(\rhobf,H,n,m)$. Recall the definition of $\mc{A}_t^{i,\varepsilon} (H)$ from \eqref{a_epsilon} and note that
\[\mc{A}_t^{i,\varepsilon} (H) = \sum_{j,k=1}^n (\partial_{\rho^{j}} \partial_{\rho^{k}}  \tilde{g}_i) (\rhobf) \mc{A}^{j,k,\varepsilon}_t (H). \]
By Cauchy-Schwarz inequality and \eqref{clp5}, the sequence $\{\mc{A}_t^{i,\varepsilon} (H), 0 \leq t \leq T\}_{\varepsilon > 0}$ is a uniform Cauchy sequence in the sense of \eqref{energy_cdondition}, equivalently, the process $\{\mc{Y}_t\}$ satisfies the $L^2$ energy condition. In particular, it has a $L^2$ limit  $\{\mc{A}_t^{i} (H), 0 \leq t \leq T\}$ as $\varepsilon \rightarrow 0$.   Passing to the limit in \eqref{clp2}, we have
\[\mc{M}_t^{i} (H) = \mc{Y}_t^{i} (H) - \mc{Y}_0^{i} (H)  - \mc{A}_t^i (H) -\lambda \int_0^t  \mc{Y}_s^{i} (\Lcal H) \,ds.\]
We conclude the proof in the case $\alpha = 3/2$ by  Lemma \ref{lem:ConvMart}.

\bibliographystyle{plain}
\bibliography{zhaoreference.bib}
\end{document}